\documentclass{article}
\usepackage[english]{babel}
\usepackage[utf8]{inputenc}
\usepackage[T1]{fontenc}
\usepackage{mathtools}   
\usepackage{amsmath}
\usepackage{amssymb}
\usepackage{amsfonts}
\usepackage{amsthm}
\usepackage{graphicx}
\usepackage{enumerate}
\usepackage{comment}
\usepackage{geometry}
\usepackage{color}
\usepackage[color,matrix,arrow]{xy}
\usepackage{xy}
\input xy
\xyoption{all}
\newtheorem{theorem}{Theorem}[section]
\newtheorem{lemma}[theorem]{Lemma}
\newtheorem{proposition}[theorem]{Proposition}

\theoremstyle{definition}
\newtheorem{definition}[theorem]{Definition}

\newtheorem{remark}[theorem]{Remark}

\begin{document}
	
\title{The coarse Novikov conjecture for extensions of coarsely embeddable groups
\thanks{Supported in part by NSFC (No. 11831006, 11771143).}}
	
\author{Qin Wang and Yazhou Zhang}
\date{}
	
\maketitle
	
\begin{abstract}
Let $(1\to N_n\to G_n\to Q_n \to 1)_{n\in \mathbb{N}}$ be a sequence of extensions of countable discrete groups. 
Endow $(G_n)_{n\in \mathbb{N}}$ with  metrics associated to proper length functions on $(G_n)_{n\in \mathbb{N}}$ respectively such that the sequence of metric spaces  $(G_n)_{n\in \mathbb{N}}$ have uniform bounded geometry. We show that if $(N_n)_{n\in \mathbb{N}}$ and $(Q_n)_{n\in \mathbb{N}}$ are coarsely embeddable into Hilbert space, then the coarse Novikov conjecture holds for the sequence $(G_n)_{n\in \mathbb{N}}$, which may not admit a coarse embedding into Hilbert space.
\end{abstract}
	
\section{Introduction}

M. Gromov introduced the notion of coarse embeddings of metric spaces into Hilbert space, and suggested that finitely generated discrete groups that are coarsely embeddable into Hilbert space, when viewed as metric spaces, might satisfy the Novikov conjecture \cite{Grom93}. G. Yu proved that
this is indeed the case \cite{Yu00}. Recall that a map $f: X\to H$ from a metric space $X$ to a Hilbert space $H$ is a {\em coarse embedding} if there exist two non-decreasing functions $\rho_1, \rho_2 :[0,\infty) \rightarrow [0,\infty)$ such that
$\lim\limits_{t\rightarrow \infty}\rho_i(t)=+\infty$ for $i=1, 2$ and such that
$\rho_1(d(x,y))\leq \|f(x)-f(y)\|\leq \rho_2(d(x,y))$
for all $x,y \in X$.
\par

 Let $(1\to N_n\to G_n\to Q_n \to 1)_{n\in \mathbb{N}}$ be a sequence of extensions of countable discrete groups. For each $n\in \mathbb{N}$, let $l_{G_n}$ be a proper length function on $G_n$. We can define a family of left invariant
 metrics $d_{G_n}$ on $G_n$ by $d_{G_n}(g, h)=l_{G_n}(g^{-1}h)$. Let $\pi: G_n\to Q_n$ be the quotient map. Define length functions on 
 $N_n$ and $Q_n$ according to
 $l_{N_n}(g)=l_{G_n}(g)$ for all $g\in N_n$, and
 $l_{Q_n}(x)=\min\{l_{G_n}(g): g\in G_n \mbox{ and } \pi(g)=x \}$
 for all $x\in Q_m$. Then these length functions induce left invariant metrics $d_{N_n}$ and $d_{Q_n}$ on $N_n$ and $Q_n$, respectively. It follows that we have three sequences discrete metric spaces $(N_n,d_{N_n})_{n\in \mathbb{N}}$, $(G_n,d_{G_n})_{n\in \mathbb{N}}$ and $(Q_n,d_{Q_n})_{n\in \mathbb{N}}$. The sequence of metric spaces $(G_n,d_{G_n})_{n\in \mathbb{N}}$ is said to have uniform bounded geometry if for every $r>0$, the number of the elements in $B_{G_n}(x,r)=\{y\in G_n: d_{G_n}(x,y)\leq r\}$ is at most $N_r$ for some $N_r>0$ independent of $n$. The sequence of metric spaces $(N_n,d_{N_n})_{n\in \mathbb{N}}$ is said to be coarsely embeddable into Hilbert space if there exists a family of maps $f_n: N_n \rightarrow H_n$, and two non-decreasing functions $\rho_1, \rho_2 :[0,\infty) \rightarrow [0,\infty)$ such that
 $\lim\limits_{t\rightarrow \infty}\rho_i(t)=+\infty$ for $i=1, 2$ and
 $\rho_1(d_{N_n}(x,y))\leq \|f_n(x)-f_n(y)\|\leq \rho_2(d_{N_n}(x,y))$
 for all $x,y \in N_n$. 

\par
The main purpose of this paper is to prove the following result:
\par
\begin{theorem}\label{main result}
Let $(1\to N_n\to G_n\to Q_n \to 1)_{n\in \mathbb{N}}$ be a sequence of extensions of countable discrete groups. Endow $(G_n)_{n\in \mathbb{N}}$ with  metrics associated to proper length functions on $(G_n)_{n\in \mathbb{N}}$ respectively such that the sequence of metric spaces  $(G_n)_{n\in \mathbb{N}}$ have uniform bounded geometry.    If $(N_n)_{n\in \mathbb{N}}$ and $(Q_n)_{n\in \mathbb{N}}$ are coarsely embeddable into Hilbert space, then the coarse Novikov conjecture holds for the sequence $(G_n)_{n\in \mathbb{N}}$.	
\end{theorem}
\par
The coarse Novikov conjecture is the injectivity part of the coarse Baum-Connes conjecture \cite{HR93,Roe93,Yu95}, and a geometric analogue of the strong Novikov conjecture \cite{BCH94}. In the case of a noncompact
complete Riemannian manifold, the coarse Novikov conjecture provides an algorithm to determine when the higher index of an elliptic operator on the
noncompact complete Riemannian manifold is nonzero. In particular, it implies that the higher index of the Dirac operator on a uniformly contractible
Riemannian manifold is nonzero. By Proposition 4.33 in \cite{Roe93}, the coarse Novikov conjecture implies Gromov's conjecture that a uniformly contractible Riemannian manifold with bounded geometry cannot have uniformly positive scalar curvature, and the zero-in-the-spectrum conjecture stating that the Laplace operator acting on the space of all $L^2$-forms of a uniformly contractible Riemannian manifold has zero in its spectrum.
See \cite{Yu06} for a comprehensive survey on the coarse Novikov conjecture.
\par
A remarkable progress was achieved by G. Yu \cite{Yu00} who proved the coarse Baum--Connes conjecture, and consequently, the coarse Novikov conjecture for metric space with bounded geometry which admits a coarse embedding into Hilbert space.  A fundamental idea underlining the
approach in \cite{Yu00} is that the index of a Dirac operator is more computable if the Dirac operator is twisted by a family of
“almost flat Bott bundles”. This approach inspires several later progresses on the coarse Novikov conjecture for coarse embeddings into
certain Banach spaces \cite{KY06,CWY15} or non-positively curved manifolds \cite{Gong-Wu-Yu20,ShanWang06}. See also \cite{Deng20,Gong-Wang-Yu08,KY12,OOY-09,Yu98} for closely related developments.
\par
In a recent paper, G. Arzhantseva and R. Tessera \cite{Arzhantseva-Tessera2018/2019} answer in the negative a long standing open question raised
in \cite{Dadarlat-Guentner 2003,Guentner-Kaminker 2004}: Does coarse embeddability into Hilbert space is preserved under group extensions of
 finitely generated groups? Their counterexamples are certain restricted permutational wreath products $\mathbb{Z}_2\wr_G H$ and
$\mathbb{Z}_2\wr_G (H\times \mathbb{F}_n)$, where $G$ is a {\em Gromov monster group} \cite{Grom03, Arzhantseva-Osajda,Osajda20}, i.e.,
a finitely generated group which contains in its Cayley graph an isometrically embedded expander \cite{Grom03}, and $H$ is a {\em Haagerup monster group} \cite{Arzhantseva-Osajda,Osajda20}, i.e., a finitely generated group with the Haagerup property but without Yu's property A \cite{Yu00,Nowak-Yu 2012}.
The reason why the groups $\mathbb{Z}_2\wr_G H$ and $\mathbb{Z}_2\wr_G (H\times \mathbb{F}_n)$ do not coarsely embed into Hilbert space is
that both groups contain isometrically in their Cayley graphs a {\em relative expander}, an innovative notion introduced by G. Arzhantseva and
R. Tessera in \cite{Arzhantseva-Tessera2015}. Consequently, the group $\mathbb{Z}_2\wr_G H$ or
$\mathbb{Z}_2\wr_G (H\times \mathbb{F}_n)$ also provide the first example of an extension of countable discrete groups or
an extension of finitely generated groups which
does not coarsely embed into any $L^p$-space for any $1\leq p<\infty$, nor into any uniformly curved Banach space, and yet does not contain
any weakly embedded expander \cite{Arzhantseva-Tessera2015}.
\par
Very recently, J. Deng, Q. Wang and G. Yu \cite{DWY21} show that if a sequence of group extensions
$( 1\to N_n\to G_n\to Q_n\to 1 )_{n\in \mathbb{N}}$
has {\em "A-by-CE" structure}, namely, the coarse disjoint union \cite{Nowak-Yu 2012} of the sequence
$\left( N_n \right)_{n\in \mathbb{N}} $ with the induced metric from the word metrics of
$\left( G_n \right)_{n\in \mathbb{N}} $ has Yu's property A, and the coarse disjoint union of the sequence $\left( Q_n \right)_{n\in \mathbb{N}} $ with the quotient metrics
coarsely embeds into Hilbert space (denoted briefly, {\em CE}), then the coarse Baum-Connes conjecture holds for the coarse disjoint union of the sequence $(G_n)_{n\in \mathbb{N}}$.
It follows that the coarse Baum-Connes conjecture, and hence the coarse Novikov conjecture, holds for the relative expanders
in \cite{Arzhantseva-Tessera2015}, the group extensions in \cite{Arzhantseva-Tessera2018/2019} mentioned above (see \cite{Braga-Chung-Li 2020,Oyono-Oyono 2001} for an alternative proof for these group extensions), and certain box spaces of free groups in \cite{Delabie-Khukhro},
which do not coarsely embed into Hilbert space and yet do not coarsely contain any weakly embedded expander. In \cite{DWY21}, J. Deng, Q. Wang and G. Yu raise an open problem: Does "CE-by-CE" imply the coarse Baum-Connes conjecture? 

Our main result Theorem \ref{main result} provides an answer to the injectivity part of this problem. Namely, "CE-by-CE" implies the coarse Novikov conjecture. Moreover, we have the following result associated to \cite[Theorem 3.13]{DWY21}, which is related to relative expanders.

\begin{theorem}\label{theorem 1.2}
	Let $(1\to N_n\to G_n\to Q_n \to 1)_{n\in \mathbb{N}}$ be a sequence of extensions of finite groups with uniformly finite generating subsets. Endow  $(N_n)_{n\in \mathbb{N}}$ with the induced metrics from the word metrics of $(G_n)_{n\in \mathbb{N}}$ and $(Q_n)_{n\in \mathbb{N}}$ with the quotient metrics from $(G_n)_{n\in \mathbb{N}}$,  respectively.  If the coarse disjoint unions $(N_n)_{n\in \mathbb{N}}$ and $(Q_n)_{n\in \mathbb{N}}$ are coarsely embeddable into Hilbert space, then the coarse Novikov conjecture holds for the coarse disjoint  union $(G_n)_{n\in \mathbb{N}}$.	
\end{theorem}
As a special case of Theorem \ref{main result}, we hve the following result for a group extension.
\begin{theorem}
	 Let $1\to N\to G\to Q \to 1 $ be a short exact sequence of countable discrete groups. If $N$ and $Q$ are coarsely embeddable into Hilbert space, then the coarse Novikov conjecture holds for $G$. 
\end{theorem}
 Note that J. Deng \cite{Deng20} has proved that the strong Novikov conjecture with arbitrary coefficient G-C*-algebras holds for CE-by-CE groups. His result bascially covers this case because G. Yu \cite{Yu95Baun-Connes} proved that the Baum--Connes conjecture for a discrete group $G$ with coefficients in $\ell^\infty(G,K(H))$  is equivalent to the coarse Baum-Connes conjecture for $G$ as a metric space with a length metric. Here our paper provides a different method for this special case.

The paper is organized as follows. In Section 2, we briefly recall the concept of the Roe algebra, the coarse Novikov conjecture and Yu's localization algebra techniques. In Section 3, we introduce the twisted Roe algebra and the twisted localization algebra for the sequence of extensions of
countable discrete groups associated to the coarse embeddings of the quotient groups $(G_m/N_m)_{m\in \mathbb{N}}$ into a Hilbert spaces. In section 4, we show that the evaluation map from the $K$-theory of the twisted localization algebra to the $K$-theory of the twisted Roe algebra is an isomorphism. In Section 5, we give the constructions of the Bott maps, and show that the Bott map from the $K$-theory of the uniform localization algebra to the $K$-theory of the twisted localization algebra is an isomorphism. Consequently, we prove the main result by recalling the result in the previous section.
\par
We should mention that the surjectivity part of the open question raised in \cite{DWY21} still remains open:
\par
{\bf Problem.} Does "CE-by-CE" imply the {\em surjectivity} of the coarse Baum-Connes assembly map for group extensions?
	
\section{The coarse Novikov conjecture for a sequence of metric spaces}
	
In this section, we shall briefly recall the Roe algebra, Yu's localization algebra and the coarse Novikov conjecture.
\par
Let $X$ be a proper metric space. Recall that a metric space is {\em proper} if every closed bounded subset is compact. An {\em $X$-module} is a separable Hilbert space equipped with a faithful  and non-degenerate $*$-representation $\pi$ of $C_0(X)$, the algebra of all complex-valued continuous functions on $X$ vanishing at infinity. An $X$-module is said to be {\em standard} if no nonzero functions in $C_0(X)$ acts as a compact operator.
When $H_X$ is an $X$-module, for each $f\in C_0(X)$ and $h\in H_X$, we simply denote $(\pi(f))h$ by $fh$.
\par
\begin{definition}\cite{Roe93,Roe96}
Let $X$ be a proper metric space, and let $H_X$ be an $X$-module.
\begin{enumerate}[(1)]
	\item  The {\em support} $\mbox{Supp}(T)$ of a bounded linear operator $T: H_X\to H_X$ is the complement in $X\times X$ of the set of
            points $(x,y) \in X\times X$ for which there exist functions $f\in C_0(X), g\in C_0(X)$ with $f(x)\neq 0, g(y)\neq 0$
            such that $gTf=0$.
	\item The {\em propagation} of a bounded linear operator $T: H_X \to H_X$ is defined to be
            $$\sup\{d(x,y):(x,y) \in \mbox{Supp}(T) \}.$$
            We say that $T$ has {\em finite propagation} if $\mbox{propagation}(T)<\infty.$
	\item A bounded linear operator $T: H_X \rightarrow H_X$ is said to be {\em locally compact} if the operators $fT$ and $Tf$ are compact
            for all $f\in C_0(X).$
\end{enumerate}
\end{definition}
\par
Denote by $\mathbb{C}[X, H_X]$, or simply $\mathbb{C}[X]$, the set of all locally compact, finite propagation operators on $H_X$.
It is straightforward to check that $\mathbb{C}[X]$ is a $*$-algebra which, up to non-canonical isomorphisms, does not depend on the choice of standard non-degenerate $X$-modules \cite{Roe96,Yu97}.
\par
\begin{definition}\cite{Roe93,Roe96}
The {\em Roe algebra} $C^*(X, H_X)$, or simply denoted by $C^*(X)$,  is defined to be the operator norm closure of $\mathbb{C}[X, H_X]$ on
a standard non-degenerate $X$-module $H_X$.
\end{definition}
\par
We may use a specific standard non-degenerate $X$-module. Let $Z$ be a countable dense subset of $X$, and let $H$ be a
fixed separable and infinite dimensional Hilbert space. Let $K:=K(H)$ denote the algebra of compact operators on $H$. The algebra $C_0(X)$ acts on $\ell^2(Z)$ by point-wise multiplications. Thus we can choose $H_X$ to be $\ell^2(Z) \otimes H$, where $C_0(X)$ acts on $H_X$ by $f(\eta \otimes h)=f\eta \otimes h $ for
$f\in C_0(X), \eta \in \ell^2(Z), h\in H$.
\par
\begin{definition}
Define $\mathbb{C}_f[X]$ to be the $*$-algebra of all bounded functions $T:Z\times Z\to {K}:=K(H)$, viewed as $Z\times Z$-matrices, such that: \\
\indent (1) for any bounded subset $B\subset X$, the set $\{(x, y)\in (B\times B)\cap (Z\times Z) | T(x, y)\neq 0 \}$ is finite; \\
\indent (2) there exists $L>0$ such that $\sharp \{y\in Z| T(x, y)\neq 0\}<L$ and $\sharp \{y\in Z|T(y, x)\neq 0\}<L$ for all $x\in Z$; \\
\indent (3) there exists $R\geq 0$ such that $T(x,y)=0$ whenever $d(x,y)>R$ for all $x, y\in Z$.
\end{definition}
\par
Note that $\mathbb{C}_f[X]$ is $*$-isomorphic to $\mathbb{C}[X, \ell^2(Z, H)]$ by a non-canonical $*$-isomorphism \cite{Yu97,Yu00}. 
Hence, we can use $\mathbb{C}_f[X]$ to replace $\mathbb{C}[X]$ to define the Roe algebra $C^*(X)$ of $X$.
\par
\begin{definition}
Let $X$ be a discrete metric space with {\em bounded geometry}, i.e., for any $R>0$ there exists $N>0$ such that the number of elements in 
every ball of radius $R$ is less than $N$. For each $d\geq 0$, the {\em Rips complex $P_d(X)$ at scale $d$} is defined to be the simplicial polyhedron where the set of vertices is $X$, and a finite subset $\{x_0, x_1, \cdots, x_q\}$ in $X$ spans a simplex if and only if
$d(x_i, x_j)\leq d$ for all $0\leq i, j\leq q$.
\end{definition}
\par
Endow $P_d(X)$ with the {\em spherical metric}. A detailed description for the spherical metric on a Rips complex can be found in
\cite[Definition 7.2.8]{WY20}. In particular, for each $d'\geq d\geq 0$, the canonical inclusion
$i_{d'd} : P_d(X) \to P_{d'}(X)$ is a coarse equivalence, and a homeomorphism onto its image \cite[Proposition 7.2.11]{WY20}.
\par
{\bf The coarse Novikov conjecture.} For any discrete metric space $X$ with bounded geometry, the coarse Baum-Connes assembly map
$$\mu: \lim\limits_{d\rightarrow \infty}K_*(P_d(X)) \rightarrow  \lim\limits_{d\rightarrow \infty}K_*(C^*(P_d(X))) \cong K_*(C^*(X))$$
is injective, where $K_*(P_d(X))=KK_*(C_0(P_d(X)), \mathbb{C})$ is the $K$-homology group \cite{Kas75} of the proper metric space $P_d(X)$ for each
$d\geq 0$.
\par
Yu's localization algebra \cite{Yu97} will play an important role in the proof of the main result of this paper.
\par
\begin{definition}\cite{Yu97,WY20} Let $X$ be a proper metric space. Define $\mathbb{C}_L[X]$ to be the $*$-algebra of all bounded and uniformly continuous functions
$$g: [0, \infty) \to \mathbb{C}_{f}[X]\subset C^*(X)$$
such that there exists a bounded function $R: [0, \infty)\to [0, \infty)$ with $\lim_{t\to \infty} R(t)=0$ and such that
$(g(t))(x, y)=0$ whenever $d(x, y)>R(t)$ for all $t\in [0, \infty)$, $x, y\in Z$.
\par
The localization algebra $C_L^*(X)$ is defined to be the norm completion of $\mathbb{C}_L[X]$ with respect to the norm
$\|g\|_\infty := \sup_{t\in [0, \infty)} \|g(t)\|$.
\end{definition}
\par
The $K$-theory of the localization algebras is invariant under strong Lipschitz homotopy invariance \cite{Yu97}.
\par
\begin{definition}\cite{Yu97}
Let $X$ and $Y$ be two proper metric spaces, and let $f$ and $g$ be two Lipschitz maps from $X$ to $Y$.  A continuous homotopy
$F:[0,1]\times X \rightarrow Y$ between $f$ and $g$ is said to be {\em strongly Lipschitz} \cite{Yu97} if \\
\indent (1) there exists $C>0$ such that $d(F(t,x),F(t,y))\leq Cd(x,y)$ for all $x,y\in X$ and $t\in [0,1]$; \\
\indent (2) for any $\varepsilon >0$, there exists $\delta >0$ such that $d(F(t_1,x),F(t_2,x))<\varepsilon$ for all $x\in X$ if
                $|t_1-t_2|<\delta$; \\
\indent (3) $F(0,x)=f(x) \text{ and } F(1,x)=g(x)$ for all $x\in X$.
\par	
A metric space $X$ is said to be {\em strongly Lipschitz homotopy equivalent} to $Y$, if there exist two Lipschitz maps $f: X\rightarrow Y$ and $g:Y\rightarrow X$ such that $f\circ g$ and $g\circ f$ are strongly Lipschitz homotopy equivalent to the identity maps $id_X$ and $id_Y$, respectively.
\end{definition}
\par
\begin{proposition}\cite{Yu97}
If $X$ is strongly Lipschitz homotopy equivalent to $Y$, then $K_*(C_L^*(X))$ is naturally isomorphic to $K_*(C_L^*(Y))$.
\end{proposition}
\par
The following Mayer-Vietoris sequence is proved in \cite{Yu97}.
\begin{proposition}\cite{Yu97}\label{six term exact sequence }
Let $X$ be a simplicial complex endowed with the spherical metric, and let $X_1$ and $X_2$ be its two simplicial subcomplexes endowed with the subspace metric. Then we have the following six term exact sequence
$$\xymatrix{
			K_0(C_L^*(X_1\cap X_2)) \ar[r] & K_0(C_L^*(X_1))\oplus K_0(C_L^*(X_2)) \ar[r] & K_0(C_L^*(X_1\cup X_2)) \ar[d]\\
			K_1(C_L^*(X_1\cup X_2)) \ar[u] & K_1(C_L^*(X_1))\oplus K_1(C_L^*(X_2))
			\ar[l] & K_1(C_L^*(X_1\cap X_2)) \ar[l]
}.$$
\end{proposition}
\par
One can define a local assembly map $\mu_L: K_*(X)\to K_*(C^*_L(X))$ \cite{Yu97}.
\par
\begin{proposition}\cite{Yu97}
For a finite dimensional simplicial polyhedron endowed with the spherical metric, the local assembly map
$\mu_L$ is an isomorphism.
\end{proposition}
\par
The {\em evaluation-at-zero map} $e: C^*_L(X) \to C^*(X)$ is defined by $e(g)=g(0)$. It induces a homomorphism at the $K$-theory level $e_*: K_*(C^*_L(X)) \to K_*(C^*(X))$. For a discrete metric space $X$ with bounded geometry, we have the following commutative diagram:
\[
\xymatrix{
& \displaystyle\lim_{d\to \infty}K_*(P_d(X))\ar[dl]_{\mu_L}^{\cong} \ar[dr]^{\mu} & \\
\displaystyle\lim_{d\to \infty} K_*(C^*_L(P_d(X)))\ar[rr]^{e_*} &  &  \displaystyle\lim_{d\to \infty} K_*(C^*(P_d(X)))
}
\]
\par
\begin{proposition}\label{evaluation}
Let $X$ be a discrete metric space with bounded geometry. To prove the coarse Novikov conjecture for $X$, it suffices to show that
$e_*$ is injective.
\end{proposition}
	
		In the rest of this section, we introduce the coarse Baum-Connes conjecture for a sequence of metric spaces.
	
	Let $(X_n)_{n\in \mathbb{N}}$ be a sequence of metric spaces with uniform bounded geometry. For each $d>0, n\in \mathbb{N}$, we choose a countable dense subset $X_d^n\subset P_d(X_n)$ such that $X_d^n\subset X_{d'}^n$ if $d<d'$.
	\begin{definition}\label{algebraic Roe algera for sequence}
		For each $d>0$, the algebraic uniform Roe algebra $\mathbb{C}_u[(P_d(X_n))_{n\in \mathbb{N}}]$ is the collection of tuples $T=\left(T^{(n)}\right)_{n\in \mathbb{N}}$ where each
		$T^{(n)}: X^n_d\times X^n_d \to K$ is a bounded founction satisfying the following conditions:
		\begin{itemize}
\item [(1)] there exists $M>0$ such that $\|T^{(n)}(x,y)\|\leq M$ for each $n\in \mathbb{N}$ and $x,y\in Z^n_d$;
\item [(2)] there exists $r>0$ such that  for each $n\in \mathbb{N}$, $T^{(n)}(x,y)=0$ for all $x,y\in X^n_d$ satisfying $d(x,y)\geq r$;	
\item [(3)] there exists $L>0$ such that for each $n\in \mathbb{N}$ and $x \in X^n_{d}$,
			$$\sharp\{y\in X_d^n:T^{(m)}(x,y)\neq 0\}\leq L, \text{ and } \sharp\{y\in X_{d}^n:T^{(m)}(y,x)\neq 0\}\leq L;$$
\item [(4)] for any sequence of uniformly bounded subsets $(B_n)_{n\in \mathbb{N}}$ where $B_n\subset P_d(X_n)$, the set $$\left\{(x,y)\in (B_n\times B_n)\cap (X^n_d\times X^n_d): T^{(n)}(x,y)\neq 0\right\}$$
			is uniformly finite.
		\end{itemize}
	\end{definition}
	
	The algebraic uniform Roe algebra $\mathbb{C}_u[(P_d(X_n))_{n\in \mathbb{N}}]$ is a $*$-algebra and admits a faithful representation on $\bigoplus\ell^2(X_d^n)\otimes H$ via multiplication. Denote $E=\bigoplus\ell^2(X_d^n)\otimes H$. The uniform Roe algebra for the sequence $(P_d(X_n))_n$, denoted by $C_u^*((P_d(X_n))_{n\in \mathbb{N}})$, the completion of $\mathbb{C}_u[(P_d(X_n))_{n\in \mathbb{N}}]$ under the operator norm on $\mathcal{B}(E)$.
	
	\begin{definition}
		The algebraic uniform localization algebra, $\mathbb{C}_{u,L}[(P_d(X_n))_{n\in \mathbb{N}}]$ is the $*$-algebra of all bounded and uniformly continuous functions $$g:[0,\infty)\rightarrow \mathbb{C}_u[(P_d(X_n))_{n\in \mathbb{N}}]$$
		such that $g(t)$ is of the form $g(t)=(g^{(n)}(t))_{n\in \mathbb{N}}$, where  the tuples $(g^{(n)}(t))_{n\in \mathbb{N}}$ satisfies the condition of Definition \ref{algebraic Roe algera for sequence} with uniform constants for all $t\in [0,\infty)$ and there is a bounded function $R: [0,\infty) \rightarrow [0,\infty)$ with $\lim\limits_{t\to \infty}R(t)=0$, such that for all  $n\in \mathbb{N}, x,y \in X_d^n$ and $t\in [1,\infty)$, $$(g^{(n)}(t))_{n\in \mathbb{N}}(x,y)=0 \text{ if } d(x,y)>R(t).$$
		
	\end{definition}
	
	The uniform localization algebra $C_{u,L}^*((P_d(X_n))_{n\in \mathbb{N}})$ is defined to be the completion of $\mathbb{C}_{u,L}[(P_d(X_n))_{n\in \mathbb{N}}]$ with repect to the norm 
	$\|g\|=\sup_{t\in [0, \infty)}\|g(t)\|.$
 
 Naturally, we have the evaluation-at-zero map $e:C_{u,L}^*((P_d(X_n))_{n\in \mathbb{N}}) \rightarrow C_{u}^*((P_d(X_n))_{n\in \mathbb{N}})$ defined by $e(g)=g(0)$.
	It induces a homomorphism
	at the K-theory level $e_*:K_*(C_{u,L}^*((P_d(X_n))_{n\in \mathbb{N}})) \rightarrow K_*(C_{u}^*((P_d(X_n))_{n\in \mathbb{N}}).$
	
	{\bf The coarse Novikov conjecture for a sequence of metric spaces.} For any sequence of discrete metric spaces $(X_n,d_n)_{n\in \mathbb{N}}$ with uniform bounded geometry, the map $$e_*:K_*(C_{u,L}^*((P_d(X_n))_{n\in \mathbb{N}})) \rightarrow K_*(C_{u}^*((P_d(X_n))_{n\in \mathbb{N}}))$$ is injective.
	
\section{Twisted algebras with coefficients from the sequence of quotient groups $(G_m/N_m)_{m\in \mathbb{N}}$  }

Let $(1\to N_m\to G_m\to G_m/N_m \to 1)_{m\in \mathbb{N}}$ be a sequence of extensions of countable discrete groups such that $(N_m)_{m\in \mathbb{N}}$ and $(G_m/N_m)_{m\in \mathbb{N}}$ are coarsely embeddable into Hilbert spaces. In this section, we will define the twisted Roe algebras and the twisted localization algebras for the sequence of Rips complexes $(P_d(G_m))_{m\in \mathbb{N}}$, in which the coefficients come from the coarse embedding of the sequence of quotient groups $(G_m/N_m)_{m\in \mathbb{N}}$ into Hilbert spaces. The construction of these twisted algebras has its
origin in \cite{Yu00}.
\par
To get started, we first recall the $C^*$-algebra associated with an infinite-dimensional Euclidean space introduced by Higson, Kasparov and Trout \cite{HKT98}. Let $V$ be a countably infinite dimensional Euclidean space. Denote by $V_a$, $V_b$, and so on, the finite dimensional affine subspaces of $V$. Denote by $V_a^0$ the finite dimensional linear subspace of $V$ consisting of differences of elements of $V_a$.
Let $\text{Cliff}(V_a^0)$ be the complexified Clifford algebra on $V_a^0$, and let $\mathcal{C}(V_a)$ be the graded $C^*$-algebra of continuous functions from $V_a$ to $\text{Cliff}(V_a^0)$ vanishing at infinity. Let $\mathcal{S}:=C_0(\mathbb{R})$ be the $C^*$-algebra of all continuous functions on $\mathbb{R}$ vanishing at infinity. Then $\mathcal{S}$ is graded according to odd and even functions. Define the graded tensor product $$\mathcal{A}(V_a)=\mathcal{S}\hat{\otimes}\mathcal{C}(V_a).$$ If $V_a \subseteq V_b$, we have a decomposition $V_b=V_{ba}^0+ V_a$, where $V_{ba}^0$ is the orthogonal complement of $V_a^0$ in $V_b^0$. For each $v_b \in V_b$, we have a unique decomposition $v_b=v_{ba}+v_a$, where $v_{ba} \in V_{ba}^0$ and $v_a\in V_a$. Every function $h$ on $V_a$ can be extended to a function $\tilde{h}$ on $V_b$ by the formula: $\tilde{h}(v_{ba}+v_a)=h(v_a)$.
\par	
\begin{definition}[\cite{HKT98}]\label{HKT Bott map} \quad (1) If $V_a \subseteq V_b$, we define $C_{ba}$ to be the Clifford algebra-valued function $V_b \rightarrow \text{Cliff}(V_b^0)$, $v_b \mapsto v_{ba}\in  V_{ba}^0 \subset \text{Cliff}(V_b^0)$. Let $X$ be the function of multiplication by $x$ on $\mathbb{R}$, considered as a degree one and unbounded multiplier of $\mathcal{S}$.  Define a homomorphism $\beta_{ba}: \mathcal{A}(V_a) \rightarrow \mathcal{A}(V_b)$ by
$$\beta_{ba}(g\hat{\otimes}h)=g(X\hat{\otimes}1+1\hat{\otimes}C_{ba})(1\hat{\otimes}\tilde{h})$$
for all $g\in \mathcal{S}$ and $h\in\mathcal{C}(V_a)$, where $g(X\hat{\otimes}1+1\hat{\otimes}C_{ba})$ is defined by the functional caculus of $g$ on the unbounded, essentially self-adjoint multipier $X\hat{\otimes}1+1\hat{\otimes}C_{ba}$. \\
\indent (2) If $V_a \subset V_b \subset V_c$, then we have $\beta_{cb} \circ \beta_{ba}=\beta_{ca}$. Hence, the above homomorphisms give rise to a directed system $(\mathcal{A}(V_a))$ as $V_a$ ranges over finite dimensional affine subspaces of $H$. We define a $C^*$-algebra $\mathcal{A}(V)$ by :
	$$\mathcal{A}(V)=\lim\limits_{\longrightarrow}\mathcal{A}(V_a),$$
where the direct limit is taken over the directed set of all finite-dimensional affine subspaces $V_a\subset V$, using the homomorphisms
$\beta_{ba}$ in (1).
\end{definition}
\par	
Now let's return back to the case of interest. Let $(1\to N_m\to G_m\to G_m/N_m \to 1)_{m\in \mathbb{N}}$ be a sequence of extensions of countable discrete groups such that $(N_m)_{m\in \mathbb{N}}$ and $(G_m/N_m)_{m\in \mathbb{N}}$ are coarsely embeddable into Hilbert spaces.  For each $m\in \mathbb{N}$, let $l_{G_m}$ be a proper length function on $G_m$, so that a left invariant metric $d_{G_m}$ on $G_m$ is defined
by $d_{G_m}(g, h)=l_{G_m}(g^{-1}h)$.  Define proper length functions on $N_m$ and $G_m/N_m$ according to $l_{N_m}(h)=l_{G_m}(h)$ for all $h\in N_m$, and
$l_{G_m/N_m}(gN_m)=\min\{l_{G_m}(gh): h\in N_m \}$ for all $gN_m\in G_m/N_m$. Denote the associated left-invariant metrics by $d_{N_m}$ and $d_{G_m/N_m}$, respectively. 
Observe that the inclusion $N_m\hookrightarrow G_m$ is an isometry, and the quotient map $G_m\to G_m/N_m$ is contractive for each $m\in \mathbb{N}$. Assume that the sequence of metric spaces  $(G_m)_{m\in \mathbb{N}}$ and $(G_m/N_m)_{m\in \mathbb{N}}$ have uniform bounded geometry.
\par
Let $f_m: G_m/N_m \rightarrow H_m$ be a family of coarse embeddings of $(G_m/N_m)_{m\in \mathbb{N}}$ into a Hilbert spaces $\{H_m\}_{m\in \mathbb{N}}$, respectively.
For each $m\in \mathbb{N}$, $n\in \mathbb{N}$, $x\in G_m/N_m$, we define $W^m_n(x)$ to be the finite dimensional Euclidean subspace of $H_m$ spanned by
$\{f_m(y): y\in G_m/N_m, d_{G_m/N_m}(x,y)\leq n^2\}$. Let $W^m(x)=\cup_{n\in \mathbb{N}} W^m_n(x)$. Then we have the following:
(1) $W^m_n(x)\subseteq W^m_{n+1}(x)$ for all $n\in \mathbb{N}$ and $x\in G_m/N_m$;  (2)  for each $r>0$, there exists $N>0$ independent of $m$ such that $W^m_n(x)\subset W^m_{n+1}(y)$ for all $n\geq N$, $x,y\in G_m/N_m$ with $d_{G_m/N_m}(x,y)\leq r$;  (3) by the uniform bounded geometry property of the sequence of metric spaces $(G_n/N_n)_{n\in \mathbb{N}}$, for each $n\in \mathbb{N}$, there exists $d_n>0$ independent of $m$ such that the dimension of $W^m_n(x)$ is less than $d_n$ for all $x\in G_m/N_m$. It follows that $W^m(x)$ is independent of the choice of $x\in G_m/N_m$. We denote
$$V_m:=W^m(x)$$
and, without loss of generality, assume that $V_m$ is dense in $H_m$.
For each $m\in \mathbb{N}$, the set $\mathbb{R}_+ \times H_m$ is equipped with the weakest topology for which the projection to $H_m$ is continuous with respect to the weak topology on $H_m$, and the function $t^2+\|h\|^2$ is continuous in $(t, h)\in \mathbb{R}_+ \times H_m$. In this way, the space $\mathbb{R}_+ \times H_m$ is a locally compact Hausdorff space. Note that, for all $v\in H_m$ and $r>0$, the set
$\mathrm{Ball}(v, r):=\{(t, h)\in \mathbb{R}_+\times H_m | t^2+\|h-v\|^2<r^2\}$ is an open subset of $\mathbb{R}_+\times H_m$. For finite dimensional affine subspaces $V_a\subset V_b$ of $V_m$, the map $\beta_{ba}$ takes $C_0(\mathbb{R}_+\times V_a)$ into $C_0(\mathbb{R}_+\times V_b)$. It follows that the $C^*$-subalgebra $\displaystyle\lim_{\to } C_0(\mathbb{R}_+\times V_a)$ of $C_0(\mathbb{R}_+\times H_m)$ is in fact $*$-isomorphic to $C_0(\mathbb{R}_+\times H_m)$. The support of an element $a\in \mathcal{A}(V_m)$ is defined to be the complement in $\mathbb{R}_+\times H_m$
of the subset of all
$(t, h)$ for which there exists $g\in C_0(\mathbb{R}_+\times H_m)$ such that $ag=0$ but $g(t, h)\neq 0$. The algebra $C_0(\mathbb{R}_+\times H_m)$
acts on $\mathcal{A}(V_m)\hat\otimes K$ by the formula: $g(a\hat\otimes k)=ga \hat\otimes k$. Hence we can define the support of an element in
$\mathcal{A}(V_m)\hat\otimes K$ in a similar way.
\par
For each $d >0$, $m \in \mathbb{N}$, the map $f_m: G_m/N_m\to H_m$ can be extended to the Rips complexes $P_d({G_m/N_m})$ as follows.
Let $G_m=\sqcup_{g\in \Lambda} \, gN_m$ be a coset decomposition of $G_m$, where $\Lambda\subset G_m$ is a set of representatives of the cosets in $G_m/N_m$.
For any point $x=\sum_{g\in \Lambda}c_g \; gN_m \in P_d(G_m/N_m)$, where all but finitely many coefficients $c_g \geq 0$ are zero and $\sum_{g\in \Lambda}c_g=1$, we define
$$f_m(x)=\sum_{g\in \Lambda}c_g f(gN_m).$$
For any point $x=\sum_{g\in \Lambda}c_g \; gN_m \in P_d(G_m/N_m)$ and $n\in \mathbb{N}$,  define $W^m_n(x)$ to be the Euclidean subspace of $V_m$ spanned by
$W^m_n(gN_m)$ for all $gN_m$ such that $c_{g}\neq 0$.
\par
Let $\pi :G_m\rightarrow G_m/N_m$ be the quotient map. It induces a map $\pi:P_d(G_m)\rightarrow P_d(G_m/N_m)$ by
$$\pi\left( \sum_{i=0}^k{c_i} g_i \right)=\sum_{i=0}^k {c_i} \pi(g_i),$$
where $c_i \geq 0$ and $\sum_{i=0}^k c_i=1$.
For each $m \in \mathbb{N}$, choose a countable dense subset $Z^m_d$ of $P_d(G_m)$ for each $d\geq 0$ such that $\pi (Z^m_d)$ is dense in $P_d(G_m/N_m)$, and $Z^m_{d} \subseteq Z^m_{d'}$ if
$d \leq d'$.
\par
\begin{definition}\label{algebraic twisted Roe algebra}
For each $d\geq 0$, the {\em algebraic uniform twisted Roe algebra $\mathbb{C}_u[(P_d(G_m),\mathcal{A}(V_m))_{m\in \mathbb{N}}]$} is defined to be the set of all tuples $T=(T^{(m)})_{m\in \mathbb{N}}$ where each
$T^{(m)}: Z^m_d\times Z^m_d \to \mathcal{A}(V_m)\hat\otimes K$ is a bounded function satisfying the following conditions:
\begin{itemize}
    \item [(1)] there exists an integer $N$ such that
                $$T^{(m)}(x,y)\in (\beta_N(\pi(x)))(\mathcal{A}(W^m_N(\pi(x))))\hat{\otimes}K\subseteq \mathcal{A}(V_m)\hat{\otimes}K$$
		          for each $m\in \mathbb{N}$ and $x,y\in Z^m_d$, where $\beta_N(\pi(x)): \mathcal{A}(W^m_N(\pi(x)))\rightarrow \mathcal{A}(V_m)$ is the $*$-homomorphism
                    associated to the inclusion of $W^m_N(\pi(x))$ into $V_m$, and $K$ is the algebra of compact operators;
	\item [(2)] there exists $M>0$ such that $\|T^{(m)}(x,y)\|\leq M$ for each $m\in \mathbb{N}$ and $x,y\in Z^m_d$;
	\item [(3)] there exists $L>0$ such that, for each $m\in \mathbb{N}$ and $y\in Z^m_d$,
		          $$\sharp\{x:T^{(m)}(x,y)\neq 0\}\leq L, \   \ \sharp\{x: T(y,x)\neq 0\}\leq L;$$
	\item [(4)] for any sequence of uniformly bounded subsets $(B_m)_{m\in \mathbb{N}}$ where $B_m\subset P_d(G_m)$, the set $$\left\{(x,y)\in (B_m\times B_m)\cap (Z^m_d\times Z^m_d): T^{(m)}(x,y)\neq 0\right\}$$
	is uniformly finite.
	\item [(5)]there exists $r_1>0$ such that for each $m\in \mathbb{N}$, if $d_{P_d(G_m)}(x,y)>r_1$, then $T^{(m)}(x,y)=0;$
	\item [(6)]there exists $r_2>0$ such that $\mathrm{Supp}(T^{(m)}(x,y))\subseteq \mathrm{Ball}(f(\pi (x)),r_2)$ for each $m\in \mathbb{N}$ and $x,y\in Z^m_d$ where
                $$\mathrm{Ball}(f(\pi(x)),r_2):=\left\{(s,h)\in \mathbb{R}_+\times H: s^2+\|h-f(\pi(x))\|^2 < r_2^2\right\}.$$
\end{itemize}
\end{definition}

	We define a product structure on $\mathbb{C}_u[(P_d(G_m),\mathcal{A}(V_m))_{m\in \mathbb{N}}]$ by:
	$$(T_1T_2)^{(m)}(x,y)=\sum_{z\in Z^m_d}T^{(m)}_1(x,z)T^{(m)}_2(z,y).$$
	Then $\mathbb{C}_u[(P_d(G_m),\mathcal{A}(V_m))_{m\in \mathbb{N}}]$ is made into a $*$-algebra via matrix multiplications and the $*$-operation on $\mathcal{A}(V_m)\hat{\otimes}K$.
Let $$E^m=\left\{\sum_{x\in Z^m_d} a_x[x]: a_x\in \mathcal{A}(V_m)\hat{\otimes}K, \sum_{x\in Z^m_d}a_x^*a_x \text{ converges in norm } \right\}.$$
Note that $E^m$ is a Hilbert $C^*$-module over $\mathcal{A}(V_m)\hat{\otimes}K$ with
	$$\left<\sum_{x\in Z^m_d} a_x[x],\sum_{x\in Z^m_d} b_x[x]\right>=\sum_{x\in Z^m_d} a_x^*b_x,$$
	$$\left(\sum_{x\in Z^m_d} a_x[x]\right)a=\sum_{x\in Z^m_d} a_xa[x]$$
for all $a\in \mathcal{A}(V_m)\hat{\otimes}K$.  We can define a $*$-representation of the $*$-algebra  $\mathbb{C}_u[(P_d(G_m),\mathcal{A}(V_m))_{m\in \mathbb{N}}]$ on $E^m$ as follows:
 $$T\left(\sum\limits_{x\in Z^m_d} a_x[x]\right)=\sum\limits_{y\in Z^m_d}\left(\sum_{x\in Z^m_d} T^{(m)}(y,x)a_x[x]\right)[y],$$
where $T=(T^{(m)})_{m\in \mathbb{N}}\in \mathbb{C}_u[(P_d(G_m),\mathcal{A}(V_m))_{m\in \mathbb{N}}] $ and $\sum\limits_{x\in Z^m_d} a_x[x] \in E^m.$ Then by taking sum of these $*$-representations, we obtain a faithful $*$-representation for the $*$-algebra $\mathbb{C}_u[(P_d(G_m),\mathcal{A}(V_m))_{m\in \mathbb{N}}]$ on $\bigoplus_{m} E^m$.
\par	
\begin{definition}\label{twisted Roe algebra}
The {\em uniform twisted Roe algebra $C_u^*((P_d(G_m),\mathcal{A}(V_m))_{m\in \mathbb{N}})$} is defined to be the operator norm closure of $\mathbb{C}_u[(P_d(G_m),\mathcal{A}(V_m))_{m\in \mathbb{N}}]$ in $\mathcal{B}(\bigoplus_m E^m)$, the $C^*$-algebra of all module homomorphisms from $\bigoplus_m E^m$ to $\bigoplus_m E^m$ for which there is an adjoint module homomorphism.
\end{definition}
\par
\begin{definition}\label{twisted localization algebra}	
Let $\mathbb{C}_{u,L}[(P_d(G_m),\mathcal{A}(V_m))_{m\in \mathbb{N}}]$ be the set of all bounded and uniformly norm-continuous functions
$$g: \mathbb{R}_+ \rightarrow \mathbb{C}_{u,L}[(P_d(G_m),\mathcal{A}(V_m))_{m\in \mathbb{N}}]$$
satisfying the following conditions:
	\begin{itemize}
		\item [(1)] there exists $N>0$ such that
                    $(g(t))(x,y)\in (\beta_N(\pi(x)))(\mathcal{A}(W^m_N(\pi(x))))\hat{\otimes}K\subseteq \mathcal{A}(V_m)\hat{\otimes}K$
                    for all $m\in \mathbb{N},t\in \mathbb{R}_+$, $x, y \in Z^m_d;$
		\item [(2)] there exists a bounded function $R:\mathbb{R}+ \rightarrow \mathbb{R}_+$ such that
                    $\lim\limits_{t\rightarrow \infty}R(t)=0$ and for all $m\in \mathbb{N}, x, y \in Z^m_d$, if $d_{P_d(G_m/N_m)}(\pi(x),\pi (y))> R(t)$, then $(g(t))(x,y)=0;$
		\item [(3)] there exists $R >0$ such that $\mathrm{Supp}(g(t)(x,y)) \subseteq \mathrm{Ball}(f(\pi(x)),R)$ for all
                    $m\in \mathbb{N}, t\in \mathbb{R}_+, x,y \in Z^m_d$.
	\end{itemize}
\par	
The {\em uniform twisted localization algebra $C_{u,L}^*((P_d(G_m),\mathcal{A}(V_m))_{m\in \mathbb{N}})$} is defined to be the norm completion of $\mathbb{C}_{u,L}[(P_d(G_m),\mathcal{A}(V_m))_{m\in \mathbb{N}}]$ with respect to the norm  $\|g\|=\sup\limits_{t\in \mathbb{R}_{+}}\|g(t)\|_{C^*((P_d(G_m),\mathcal{A}(V_m))_{m\in \mathbb{N}})}.$
\end{definition}
\begin{remark}
	For each $m\in \mathbb{N}$, we can define algebraic twisted Roe algebra $\mathbb{C}[(P_d(G_m),\mathcal{A}(V_m))]$ consisting of all bounded functions $T^{(m)}: Z^m_d\times Z^m_d \to \mathcal{A}(V_m)\hat\otimes K$ satisfying the conditions of Definition \ref{algebraic twisted Roe algebra}. Then the $*$-algebra $\mathbb{C}_u[(P_d(G_m),\mathcal{A}(V_m))_{m\in \mathbb{N}}]$ can be regarded as $\prod^u\limits_{m\in \mathbb{N}}\mathbb{C}[P_d(G_m),\mathcal{A}(V_m)]$, where we add $u$ on the direct product symbol to emphasize that the latter $*$-algabra consists of  the family of operators $T^{(m)}$, $m\in \mathbb{N}$, satisfy the conditions of
	Definition \ref{algebraic twisted Roe algebra} with uniform constants. Denote by $\prod^u\limits_{m\in \mathbb{N}}C^*(P_d(G_m),\mathcal{A}(V_m))$, is the operator norm closure of $\prod^u\limits_{m\in \mathbb{N}}\mathbb{C}[P_d(G_m),\mathcal{A}(V_m)]$ in $\mathcal{B}(\oplus_m E^m)$. Naturally, the $C^*$-algebra $\prod^u\limits_{m\in \mathbb{N}}C^*(P_d(G_m),\mathcal{A}(V_m))$ is equal to $C_u^*((P_d(G_m),\mathcal{A}(V_m))_{m\in \mathbb{N}})$. We can similarly define the $C^*$-algebra $\prod^u\limits_{m\in \mathbb{N}}C^*_L(P_d(G_m),\mathcal{A}(V_m))$ as another form of $C_{u,L}^*((P_d(G_m),\mathcal{A}(V_m))_{m\in \mathbb{N}})$.
\end{remark}

\par
There is a natural {\em evaluation-at-zero} homomorphism
$$e: C_{u,L}^*((P_d(G_m),\mathcal{A}(V_m))_{m\in \mathbb{N}}) \rightarrow C_u^*((P_d(G_m),\mathcal{A}(V_m))_{m\in \mathbb{N}})$$
defined by $e(g)=g(0)$. It induces a homomorphism at the $K$-theory level:
$$e_*: \lim\limits_{d\rightarrow \infty}K_*(C_{u,L}^*((P_d(G_m),\mathcal{A}(V_m))_{m\in \mathbb{N}})) \rightarrow \lim\limits_{d\rightarrow \infty}K_*(C_u^*((P_d(G_m),\mathcal{A}(V_m))_{m\in \mathbb{N}})).$$
	
\section{Reduction to the sequence of normal subgroups $(N_m)_{m\in\mathbb{N}}$}

In this section, we shall prove that the {\em "twisted coarse Baum-Connes conjecture for the sequence $(G_m)_{m\in \mathbb{N}}$"} holds. Namely, we have the following result.
\par
\begin{theorem}\label{twisted iso}
Let $(1\to N_m\to G_m\to G_m/N_m \to 1)_{m\in \mathbb{N}}$ be an extension of countable discrete groups. Assume that the sequence of metric spaces $(G_m)_{m\in \mathbb{N}}$ have uniform bounded geometry.  If $(N_m)_{m\in \mathbb{N}}$ and $(G_m/N_m)_{m\in \mathbb{N}}$ are coarsely embeddable into Hilbert space, then the evaluation-at-zero homomorphism on $K$-theory  		
$$e_*:\lim\limits_{d\rightarrow \infty}K_*(C_{u,L}^*((P_d(G_m),\mathcal{A}(V_m))_{m\in \mathbb{N}})) \rightarrow \lim\limits_{d\rightarrow \infty}K_*(C_u^*((P_d(G_m),\mathcal{A}(V_m))_{m\in \mathbb{N}}))$$
is an isomorphism.
\end{theorem}
The proof proceeds by decomposing the twisted algebras into various smaller ideals or subalgebras, and applying a Mayer-Vietoris sequence argument and the strong Lipschitz homotopy invariance for these subalgebras in the twisted localization algebras to reduce the twisted coarse Baum-Connes conjecture for the sequene $(G_n)_{n\in \mathbb{N}}$ to the coarse Baum-Connes conjecture for sequence of the normal subgroups $(N_m)_{m\in \mathbb{N}}$ with {\em non-twisted coefficients}, which holds by appealing to Yu's original work in \cite{Yu00} or the recent work \cite{DWY21} of J. Deng, Q. W and G. Yu for the sequence $(N_m)_{m\in \mathbb{N}}$. 	
\par
To do so, we first discuss ideals of the uniform twisted algebras associated with any family of open subsets of $\{\mathbb{R}_+ \times H_m\}_{m\in \mathbb{N}}$.
\par	
	\begin{definition}
		\begin{itemize}
			\item [(1)] The support of an element $T=(T^{(m)})_{m\in \mathbb{N}} \in \mathbb{C}_u[(P_d(G_m),\mathcal{A}(V_m))_{m\in \mathbb{N}}]$ is defined to be a sequence of subsets
			$$\mathrm{Supp}(T):=(\mathrm{Supp}(T^{(m)}))_{m\in \mathbb{N}}$$ where
                            $$\mathrm{Supp}(T^{(m)}):=\left\{(x,y,t,h)\in Z^m_d\times Z^m_d \times \mathbb{R}_+ \times H_m: (t,h)\in \mathrm{Supp}(T^{(m)}(x,y))\right\}.$$
			\item [(2)] The support of an element $g$ in $\mathbb{C}_{u,L}[(P_d(G_m),\mathcal{A}(V_m))_{m\in \mathbb{N}}]$ is defined to be
                            $\mathop{\cup}\limits_{t\in  \mathbb{R}_+} \mathrm{Supp}(g(t)).$				
		\end{itemize}
	\end{definition}
\par	
Let $\mathnormal{O}=(O_m)_{m\in \mathbb{N}}$ be a sequence of open subsets, where each $O_m$ is an open subset in $\mathbb{R}_+ \times H_m$. Define $\mathbb{C}_u[(P_d(G_m),\mathcal{A}(V_m))_{m\in \mathbb{N}}]_{\mathnormal{O}}$ to be the subalgebra of $\mathbb{C}_u[(P_d(G_m),\mathcal{A}(V_m))_{m\in \mathbb{N}}]$ consisting of all elements $T=(T^{(m)})_{m\in \mathbb{N}}$ satisfying $\mathrm{Supp}(T^{(m)})\subset Z^m_d\times Z^m_d \times \mathnormal{O_m}$ for each $m\in \mathbb{N}$, i.e.,
$$\mathbb{C}_u[(P_d(G_m),\mathcal{A}(V_m))_{m\in \mathbb{N}}]_{\mathnormal{O}}=\left\{T\in \mathbb{C}_u[(P_d(G_m),\mathcal{A}(V_m))_{m\in \mathbb{N}}]:
\mathrm{Supp}(T^{(m)}(x,y))\subseteq \mathnormal{O_m}, \forall m\in \mathbb{N}, x,y\in Z^m_d\right\}.$$
Define $C^*$-subalgebra $C_u^*((P_d(G_m),\mathcal{A}(V_m))_{m\in \mathbb{N}})_{\mathnormal{O}}$ to be the norm closure of $\mathbb{C}_u[(P_d(G_m),\mathcal{A}(V_m))_{m\in \mathbb{N}}]_{\mathnormal{O}}$
in $C_u^*((P_d(G_m),\mathcal{A}(V_m))_{m\in \mathbb{N}})$. similarly, We can also define $C^*$-subalgebra $C_{u,L}^*((P_d(G_m),\mathcal{A}(V_m))_{m\in \mathbb{N}})_{\mathnormal{O}}$. Note that $C_u^*((P_d(G_m),\mathcal{A}(V_m))_{m\in \mathbb{N}})_{\mathnormal{O}}$ and $C_{u,L}^*((P_d(G_m),\mathcal{A}(V_m)_{\mathnormal{O}})_{m\in \mathbb{N}})$ are closed two-sided ideals of $C_u^*((P_d(G_m),\mathcal{A}(V_m))_{m\in \mathbb{N}})$ and $C_{u,L}^*((P_d(G_m),\mathcal{A}(V_m))_{m\in \mathbb{N}})$, respectively.
\par	
We have the following lemma analogous to \cite[Lemma 6.3]{Yu00}.
\begin{lemma}\label{ideal decomp} For each $m\in\mathbb{N}, v_m\in V_m$, r>0, 
	 let $\mathrm{Ball}(v_m,r)=\left\{(t,h)\in \mathbb{R}_+ \times H_m: t^2+\|h-v_m)\|< r^2\right\}$, and $X^m_{i}$, $1\leq i\leq i_0$, be mutually disjoint subsets of $G_m/N_m$. Fix arbitrarily $1< j\leq i_0$.
	Let $\mathnormal{O_r}=(O_{m,r})_{m\in \mathbb{N}}$ and $\mathnormal{O'_r}=(O'_{m,r})_{m\in \mathbb{N}}$ be two sequences of open subsets,  in which for each $ m\in \mathbb{N}$, $$\mathnormal{O}_{m,r}=\mathop\bigcup_{i=1}^{j-1}\left( \mathop \bigcup\limits_{x \in X^m_{i}}\mathrm{Ball}\big( f(x),r \big) \right) \mbox{ \textrm{and} } \mathnormal{O}'_{m,r}=\mathop \bigcup\limits_{x \in X^m_{j}}\mathrm{Ball}\big(f(x),r\big),$$
where $f_m:G_m/N_m\to H_m$ is the coarse embedding. Then for each $r_0>0$ we have
\begin{enumerate}[(1)]
\item  $\lim\limits_{r<r_0,r\to r_0} C_u^*((P_d(G_m),\mathcal{A}(V_m))_{m\in \mathbb{N}})_{\mathnormal{O}_{r}\cup \mathnormal{O}_r'}\\
            =\lim\limits_{r<r_0,r\to r_0}C_u^*((P_d(G_m),\mathcal{A}(V_m))_{m\in \mathbb{N}})_{\mathnormal{O}_r}+\lim\limits_{r<r_0,r\to r_0}C_u^*((P_d(G_m),\mathcal{A}(V_m))_{m\in \mathbb{N}})_{\mathnormal{O}_r'}$;
\item $\lim\limits_{r<r_0,r\to r_0}C_u^*((P_d(G_m),\mathcal{A}(V_m))_{m\in \mathbb{N}})_{\mathnormal{O}_r\cap \mathnormal{O}_r'}\\
            =\lim\limits_{r<r_0,r\to r_0} \big( C_u^*((P_d(G_m),\mathcal{A}(V_m))_{m\in \mathbb{N}})_{\mathnormal{O}_r}\cap C_u^*((P_d(G_m),\mathcal{A}(V_m))_{m\in \mathbb{N}}\big)_{\mathnormal{O}_r'} $;
\item $\lim\limits_{r<r_0,r\to r_0}C_{u,L}^*((P_d(G_m),\mathcal{A}(V_m))_{m\in \mathbb{N}})_{\mathnormal{O}_r\cup \mathnormal{O}_r'}\\
            =\lim\limits_{r<r_0,r\to r_0}C_{u,L}^*((P_d(G_m),\mathcal{A}(V_m))_{m\in \mathbb{N}})_{\mathnormal{O}_r}+\lim\limits_{r<r_0,r\to r_0}C_{u,L}^*((P_d(G_m),\mathcal{A}(V_m))_{m\in \mathbb{N}})_{\mathnormal{O}_r'}$;
\item  $\lim\limits_{r<r_0,r\to r_0}C_{u,L}^*((P_d(G_m),\mathcal{A}(V_m))_{m\in \mathbb{N}})_{\mathnormal{O}_r\cap \mathnormal{O}_r'}\\
            =\lim\limits_{r<r_0,r\to r_0} \big( C_{u,L}^*((P_d(G_m),\mathcal{A}(V_m))_{m\in \mathbb{N}})_{\mathnormal{O}_r}\cap C_{u,L}^*((P_d(G_m),\mathcal{A}(V_m))_{m\in \mathbb{N}}\big) _{\mathnormal{O}_r'} $.
\end{enumerate}
\end{lemma}
\par
It would be convenient to introduce the following terminology.	
\par
\begin{definition}
Let $r>0$. For each $m\in \mathbb{N}$, a family of open subsets $\{O_{m,r}(i)\}_{i\in J}$ of $\mathbb{R}_+\times H_m$ is said to be {\em $(G_m/N_m,r)$-separate} if: \\
\indent (1) $\mathnormal{O}_{m,r}(i) \cap \mathnormal{O}_{m,r}(j) =\emptyset$ if $i\neq j$; \\
\indent (2) for each $i\in J$, there exists $x_i^m\in G_m$ such that $\mathnormal{O}_{m,r}(i) \subseteq \mathrm{Ball}(f_m(\pi(x_i)),r)$, where
            $f$ is the coarse embedding $f:G_m/N_m\rightarrow H_m$ and
                $$\mathrm{Ball}(f_m(\pi(x_i^m)),r)=\left\{(t,h)\in \mathbb{R}_+\times H_m: t^2+\|h-f_m(\pi(x_i^m))\|^2<r^2\right\}. $$
                
\end{definition}

\par
The following proposition contains a key step in reducing the twisted coarse Baum-Connes conjecture for the sequence $(G_m)_{m\in \mathbb{N}}$ to the coarse Baum-Connes conjecture for the sequence $(N_m)_{m\in \mathbb{N}}$ with non-twisted coefficients.
\par
\begin{proposition}\label {local iso}
Let $r>0$. If $\mathnormal{O_r}=(O_{m,r})_{m\in \mathbb{N}}$ be a sequence of open subsets, where each $O_{m,r}$ is a family of $(G_m/N_m,r)$-separate open subsets of $\mathbb{R}_+\times H_m$. Then
$$e_*: \lim\limits_{d\rightarrow \infty}K_*(C_{u,L}^*((P_d(G_m),\mathcal{A}(V_m))_{m\in \mathbb{N}})_{\mathnormal{O_r}})
        \rightarrow \lim\limits_{d\rightarrow \infty}K_*(C_u^*\left((P_d(G_m),\mathcal{A}(V_m))_{m\in \mathbb{N}}\right)_{\mathnormal{O_r}}) $$
is an isomorphism.
\end{proposition}
\begin{proof}
We first discuss the structure of the $C^*$-algebra $C_u^*\left((P_d(G_m),\mathcal{A}(V_m))_{m\in \mathbb{N}}\right)_{\mathnormal{O_r}}$.
\par		
Let $T=(T^{(m)})_{m\in \mathbb{N}} \in \mathbb{C}_u[(P_d(G_m),\mathcal{A}(V_m))_{m\in \mathbb{N}}]$. For each $m\in \mathbb{N}$, since $\mathnormal{O_{m,r}}$ is the union of a family of  disjoint open subsets $\{\mathnormal{O}_{m,r}(i)\}_{i\in J}$ of $\mathbb{R}_+\times H_m$, the operator $T^{(m)}$ has a formal decomposition
$T^{(m)}=\sum_{i\in J} T^{(m)}_i$, where $T^{(m)}_i:=T^{(m)}|_{O_{m,r}(i)}$ is the restriction of $T^{(m)}$ on the component $O_{m,r}(i)$, so that $\mathrm{Supp}(T^{(m)}_i)\subseteq Z^m_d\times Z^m_d\times\mathnormal{O}_{m,r}(i)$.
It follows that the correspondence
$$T=(T^{(m)})_{m\in \mathbb{N}} \mapsto ((T^{(m)}_i)_{i\in J})_{m\in \mathbb{N}}$$
gives rise to a $*$-isomorphism
$$\mathbb{C}_u[(P_d(G_m),\mathcal{A}(V_m))_{m\in \mathbb{N}}]\cong \prod^u_{m\in \mathbb{N}}\prod_{i\in J}^u
		\mathbb{C}[P_d(G_m),\mathcal{A}(V_m)]_{\mathnormal{O_{m,r}(i)}}\cong \prod^u_{\substack{i\in J \\ m\in \mathbb{N}}}
		\mathbb{C}[P_d(G_m),\mathcal{A}(V_m)]_{\mathnormal{O_{m,r}(i)}}$$
where we add $u$ on the direct product symbol to emphasize that the family of operators $T^{(m)}_i$, $i\in J, m\in \mathbb{N}$, satisfy the conditions of
Definition \ref{algebraic twisted Roe algebra} with uniform constants.
\par
Since $\{O_{m,r}(i)\}_{i\in J}$ is a family of $(G_m/N_m,r)$-separate open subsets of $\mathbb{R}_+\times H_m$,
there exists $x_i^m\in G_m$ such that $\mathnormal{O}_{m,r}(i) \subseteq \mathrm{Ball}(f_m(\pi(x^m_i)),r)$ for each $i\in J$.
Note that the condition (6) in Definition \ref{algebraic twisted Roe algebra} and the coarse embedding property of the sequence $(G_m/N_m)_{m\in \mathbb{N}}$ imply that for each $m\in \mathbb{N}$, $d>0$ and $i\in J$,
$$C^*(P_d(G_m),\mathcal{A}(V_m))_{\mathnormal{O_{m,r}(i)}}=\lim\limits_{R\rightarrow \infty}
C^*\Big( P_d\big( \mathcal{N}_{G_m}(x_i^mN_m,R)\big), \mathcal{A}(V_m)\Big)_{\mathnormal{O_{m,r}(i)}},$$
where $\mathcal{N}_{G_m}(x_i^mN_m,R)$ denotes the $R$-neighborhood of the coset $x_i^m N_m$ in $G_m$. More concretely,
for each $i\in J$ and $R>0$, there exists
an integer $N_R>0$, and $g^{(i)}_1, g^{(i)}_2, \cdots, g^{(i)}_{N_R}$ in $G_m$, such that
$$ \mathcal{N}_{G_m}(x_i^mN_m,R)=\mathop{\sqcup}\limits_{k=1}^{N_R}g_k^{(i)}N_m, $$
where $g_k^{(i)}N_m \in G_m/N_m$ satisfies $d_{G_m/N_m}(\pi(g_k^{(i)}),\pi(x_i^m))< R$. In other words, for each $k$ there exists $h_k^{(i)}\in N_m$ such that $d_G(g_k^{(i)},x_i^mh_k^{(i)})< R$.
Note that for each fixed $R>0$, the uniform bounded geometry property of the sequence $(G_m/N_m)_{m\in \mathbb{N}}$  implies that the integer $N_R$ is independent of $m\in \mathbb{N}$ and $i\in J$. Sincce the family of coarse embeddings $f_m:G_m/N_m \rightarrow H_m$ are controlled by two common non-decreasing functions,
it follows that 
$$C_u^*((P_d(G_m),\mathcal{A}(V_m))_{m\in \mathbb{N}})_{\mathnormal{O_r}}=\lim\limits_{R\rightarrow \infty}\prod^u_{\substack{i\in J \\ m\in \mathbb{N}}}
C^*\Big( P_d \big( \mathcal{N}_{G_m}(x_i^mN_m,R)\big) ,\mathcal{A}(V_m)\Big)_{\mathnormal{O_{m,r}(i)}}$$
and
$$C_{u,L}^*\left((P_d(G_m),\mathcal{A}(V_m))_{m\in \mathbb{N}}\right)_{\mathnormal{O_r}}=\lim\limits_{R\rightarrow \infty}\prod^u_{\substack{i\in J \\ m\in \mathbb{N}}} C_{L}^*\Big(P_d\big(\mathcal{N}_{G_m}(x_i^mN_m,R)\big),\mathcal{A}(V_m)\Big)_{\mathnormal{O_{m,r}(i)}}.$$	
Hence, to prove the proposition, we only need to prove that
\begin{equation*}
	\begin{aligned}	
e_*: \lim\limits_{d\rightarrow \infty}\lim\limits_{R\rightarrow \infty} K_*\Bigg(\prod^u_{\substack{i\in J \\ m\in \mathbb{N}}} C_{L}^*\Big(P_d\big(\mathcal{N}_{G_m}(x_i^mN_m &,R)\big),\mathcal{A}(V_m)\Big)_{\mathnormal{O_{m,r}(i)}}\Bigg)
\rightarrow\\
 &\lim\limits_{d\rightarrow \infty} \lim\limits_{R\rightarrow \infty}K_*\Bigg(\prod^u_{\substack{i\in J \\ m\in \mathbb{N}}}
C^*\Big( P_d \big( \mathcal{N}_{G_m}(x_i^mN_m,R)\big) ,\mathcal{A}(V_m)\Big)_{\mathnormal{O_{m,r}(i)}}\Bigg)
\end{aligned}
\end{equation*}
is an isomorphism.
\par		
Note that for any fixed $d>0$, if $R<R'$ then $\mathcal{N}_{G_m}(x_i^mN_m,R)\subseteq \mathcal{N}_{G_m}(x_i^mN_m,R')$ for each $m\in \mathbb{N}$ and $i\in J$, so that we have natural inclusions
$$\prod^u_{\substack{i\in J \\ m\in \mathbb{N}}}
C^*\Big( P_d \big( \mathcal{N}_{G_m}(x_i^mN_m,R)\big) ,\mathcal{A}(V_m)\Big)_{\mathnormal{O_{m,r}(i)}} \subseteq
\prod^u_{\substack{i\in J \\ m\in \mathbb{N}}}
C^*\Big( P_d \big( \mathcal{N}_{G_m}(x_i^mN_m,R')\big) ,\mathcal{A}(V_m)\Big)_{\mathnormal{O_{m,r}(i)}}$$
and
$$\prod^u_{\substack{i\in J \\ m\in \mathbb{N}}} C_{L}^*\Big(P_d\big(\mathcal{N}_{G_m}(x_i^mN_m,R)\big),\mathcal{A}(V_m)\Big)_{\mathnormal{O_{m,r}(i)}} \subseteq
\prod^u_{\substack{i\in J \\ m\in \mathbb{N}}} C_{L}^*\Big(P_d\big(\mathcal{N}_{G_m}(x_i^mN_m,R')\big),\mathcal{A}(V_m)\Big)_{\mathnormal{O_{m,r}(i)}}.$$
On the other hand, for any fixed $R>0$, if $d<d'$, we have natural inclusions
$$\prod^u_{\substack{i\in J \\ m\in \mathbb{N}}} C^*\Big(P_d\big(\mathcal{N}_{G_m}(x_i^mN_m,R)\big),\mathcal{A}(V_m)\Big)_{\mathnormal{O_{m,r}(i)}} \subseteq \prod^u_{\substack{i\in J \\ m\in \mathbb{N}}} C^*\Big(P_{d'}\big(\mathcal{N}_{G_m}(x_i^mN_m,R)\big),\mathcal{A}(V_m)\Big)_{\mathnormal{O_{m,r}(i)}}$$ and $$\prod^u_{\substack{i\in J \\ m\in \mathbb{N}}} C_{L}^*\Big(P_d\big(\mathcal{N}_{G_m}(x_i^mN_m,R)\big),\mathcal{A}(V_m)\Big)_{\mathnormal{O_{m,r}(i)}} \subseteq \prod^u_{\substack{i\in J \\ m\in \mathbb{N}}} C_{L}^*\Big(P_{d'}\big(\mathcal{N}_{G_m}(x_i^mN_m,R)\big),\mathcal{A}(V_m)\Big)_{\mathnormal{O_{m,r}(i)}}.$$
Therefore, we can change the order of inductive limits on $K$-theory groups:
\begin{equation*}
\begin{split}
 \lim\limits_{d\rightarrow \infty}\lim\limits_{R\rightarrow \infty} K_*\Bigg(\prod^u_{\substack{i\in J \\ m\in \mathbb{N}}} C^*\Big(P_d\big(\mathcal{N}_{G_m}(x_i^mN_m,&R)\big), \mathcal{A}(V_m)\Big)_{\mathnormal{O_{m,r}(i)}}\Bigg)
\cong\\
& \lim\limits_{R\rightarrow \infty} \lim\limits_{d\rightarrow \infty}K_*\Bigg(\prod^u_{\substack{i\in J \\ m\in \mathbb{N}}}
C^*\Big( P_d \big( \mathcal{N}_{G_m}(x_i^mN_m,R)\big) ,\mathcal{A}(V_m)\Big)_{\mathnormal{O_{m,r}(i)}}\Bigg)
\end{split}
\end{equation*}
and

\begin{equation*}
\begin{split}
 \lim\limits_{d\rightarrow \infty}\lim\limits_{R\rightarrow \infty} K_*\Bigg(\prod^u_{\substack{i\in J \\ m\in \mathbb{N}}} C_{L}^*\Big(P_d\big(\mathcal{N}_{G_m}(x_i^mN_m,&R)\big), \mathcal{A}(V_m)\Big)_{\mathnormal{O_{m,r}(i)}}\Bigg)
\cong\\
& \lim\limits_{R\rightarrow \infty} \lim\limits_{d\rightarrow \infty}K_*\Bigg(\prod^u_{\substack{i\in J \\ m\in \mathbb{N}}}
C_L^*\Big( P_d \big( \mathcal{N}_{G_m}(x_i^mN_m,R)\big) ,\mathcal{A}(V_m)\Big)_{\mathnormal{O_{m,r}(i)}}\Bigg)
\end{split}
\end{equation*}
\par
As a result, to prove the proposition, it suffices to prove that, for each fixed $R>0$,

\begin{equation*}
\begin{split}
e_*:  \lim\limits_{d\rightarrow \infty} K_*\Bigg(\prod^u_{\substack{i\in J \\ m\in \mathbb{N}}} C_{L}^*\Big(P_d\big(\mathcal{N}_{G_m}(x_i^mN_m,R)\big),&\mathcal{A}(V_m)\Big)_{\mathnormal{O_{m,r}(i)}}\Bigg)
\rightarrow\\
& \lim\limits_{d\rightarrow \infty}K_*\Bigg(\prod^u_{\substack{i\in J \\ m\in \mathbb{N}}}
C^*\Big( P_d \big( \mathcal{N}_{G_m}(x_i^mN_m,R)\big) ,\mathcal{A}(V_m)\Big)_{\mathnormal{O_{m,r}(i)}}\Bigg)
\end{split}
\end{equation*}
is an isomorphism.
	
Note that when $d$ is large enough, for each $m\in \mathbb{N}$, $P_d(\mathcal{N}_{G_m}(x_i^mN_m,R))$ as a subcomplex of $P_d(G_m)$ is strongly Lipschitz homotopy equivalent to
$P_d(x_i^mN_m)$ by a linear homotopy. More precisely, for each $h\in N_m$, there exists $\widetilde{h}\in N_m$ such that
$$d_{P_d(G_m)}(g_k^{(i)}h, x_i \widetilde{h} h_k^{(i)})=d_{P_d(G_m)}(g_k^{(i)}, x_i h_k^{(i)})<R$$
for each $g_k^{(i)}N\subseteq \mathcal{N}_{G_m}(x_i^mN_m,R)$, so that we can connect
$g_k^{(i)}h\in g_k^{(i)}N_m$ with $x_i \widetilde{h} h_k^{(i)}\in x_i N_m$ by a line segment for each $h\in N_m$. Since $x_i N_m$ is isometric to $N_m$,
it follows that $P_d(\mathcal{N}_{G_m}(x_i^mN_m,R))$ as a subcomplex of $P_d(G_m)$ is strongly Lipschitz homotopy equivalent to
$P_d(N_m)$ for all $m\in \mathbb{N}$ and $i\in J$. Moreover, the constants controlling the strong Lipschitz homotopy equivalences are independent of $m\in \mathbb{N}$ and $i\in J$,
because the metric on the Rips complex $P_d(G_m)$ is symmetric/homogeneous.
It follows from the strong Lipschitz homotopy invariance for $K$-theory of Roe algebras and localization algebras
that the inclusion maps $P_d(x_i^mN_m) \rightarrow P_d(\mathcal{N}_{G_m}(x_i^mN_m,R))$ for each $m\in \mathbb{N}$ induce the following commutative diagram
\begin{small}

\begin{equation*}
\xymatrix{
\lim\limits_{d\rightarrow \infty} K_*\Big(\prod^u\limits_{\substack{i\in J \\ m\in \mathbb{N}}} C_{L}^*(P_d(N_m)),\mathcal{A}(V_m))_{\mathnormal{O_{m,r}(i)}}\Big) \ar[d]^{\cong}\ar[r]^{e_*} & \lim\limits_{d\rightarrow \infty} K_*\Big(\prod^u\limits_{\substack{i\in J \\ m\in \mathbb{N}}} C^*(P_d(N_m)),\mathcal{A}(V_m))_{\mathnormal{O_{m,r}(i)}}\Big)\ar[d]^\cong\\
\lim\limits_{d\rightarrow \infty} K_*\Big(\prod^u\limits_{\substack{i\in J \\ m\in \mathbb{N}}} C_{L}^*(P_d(\mathcal{N}_{G_m}(x_i^mN_m)),\mathcal{A}(V_m))_{\mathnormal{O_{m,r}(i)}}\Big) \ar[d]^{\cong}\ar[r]^{e_*} & \lim\limits_{d\rightarrow \infty} K_*\Big(\prod^u\limits_{\substack{i\in J \\ m\in \mathbb{N}}} C^*(P_d(\mathcal{N}_{G_m}(x_i^mN_m)),\mathcal{A}(V_m))_{\mathnormal{O_{m,r}(i)}}\Big) \ar[d]^{\cong}\\
\lim\limits_{d\rightarrow \infty} K_*\Big(\prod^u\limits_{\substack{i\in J \\ m\in \mathbb{N}}} C_{L}^*(P_d(\mathcal{N}_{G_m}(x_i^mN_m,R)),\mathcal{A}(V_m))_{\mathnormal{O_{m,r}(i)}}\Big)\ar[r]^{e_*} & \lim\limits_{d\rightarrow \infty} K_*(\prod^u\limits_{\substack{i\in J \\ m\in \mathbb{N}}} C^*(P_d(\mathcal{N}_{G_m}(x_i^mN_m,R)),\mathcal{A}(V_m))_{\mathnormal{O_{m,r}(i)}}\Big)
}
\end{equation*}
\end{small}
where the vertical maps are isomorphisms by strongly Lipschitz homotopy invariance.
\par
Therefore, to prove the proposition, it suffices to prove that
$$ e_*: \lim\limits_{d\rightarrow \infty} K_*\Bigg(\prod^u\limits_{\substack{i\in J \\ m\in \mathbb{N}}} C_{L}^*(P_d(N_m)),\mathcal{A}(V_m))_{\mathnormal{O_{m,r}(i)}}\Bigg) \rightarrow \lim\limits_{d\rightarrow \infty} K_*\Bigg(\prod^u\limits_{\substack{i\in J \\ m\in \mathbb{N}}} C^*(P_d(N_m)),\mathcal{A}(V_m))_{\mathnormal{O_{m,r}(i)}}\Bigg)$$
is an isomorphism.
\par		
Note that we have now arrived at the Roe algebras of $P_d(N_m)$ with non-twisted
coefficients $\prod\limits_{i\in J}\mathcal{A}_{\mathnormal{O_{m,r}(i)}}$. Namely, we have
$$\prod^u\limits_{\substack{i\in J \\ m\in \mathbb{N}}} C^*(P_d(N_m)),\mathcal{A}(V_m))_{\mathnormal{O_{m,r}(i)}}\cong
\prod^u\limits_ {m\in \mathbb{N}} C^*(P_d(N_m)),\prod\limits_{i\in J}\mathcal{A}(V_m)_{\mathnormal{O_{m,r}(i)}}) $$
and
$$\prod^u\limits_{\substack{i\in J \\ m\in \mathbb{N}}} C_{L}^*(P_d(N_m)),\mathcal{A}(V_m))_{\mathnormal{O_{m,r}(i)}}\cong
\prod^u\limits_{m\in \mathbb{N}} C_{L}^*(P_d(N_m)),\prod\limits_{i\in J}\mathcal{A}(V_m)_{\mathnormal{O_{m,r}(i)}})$$
where $\prod\limits_{i\in J}\mathcal{A}_{\mathnormal{O_{m,r}(i)}}=\left\{(a_i)_{i\in J}: a_i\in \mathcal{A}_{\mathnormal{O_{m,r}(i)}}, \sup\limits_{i\in J}\|a_i\|< +\infty\right\}$ for each $m\in \mathbb{N}$.
\par		
Since the sequence of  normal subgroups $(N_m)_{m\in \mathbb{N}}$ with the proper metric $d_{N_m}$ induced from the sequence $(G_m)_{m\in \mathbb{N}}$ is a sequence of metric spaces with uniform bounded geometry, and admits a coarse
embedding into Hilbert space by assumption, we can appeal to the same proof of the celebrate Theorem 1.1 of G. Yu \cite{Yu00} to conclude that
$$e_*: \lim\limits_{d \rightarrow \infty}K_*\Bigg( \prod^u\limits_{m\in \mathbb{N}} C_{L}^*(P_d(N_m)),\prod\limits_{i\in J}\mathcal{A}(V_m)_{\mathnormal{O_{m,r}(i)}}) \Bigg)
\rightarrow \lim\limits_{d \rightarrow \infty}
K_*\Bigg(\prod^u\limits_{m\in \mathbb{N}} C^*(P_d(N_m)),\prod\limits_{i\in J}\mathcal{A}(V_m)_{\mathnormal{O_{m,r}(i)}}) \Bigg) $$
is an isomorphism.	Consequently, we have that
$$e_*: \lim\limits_{d\rightarrow \infty}K_*(C_{u,L}^*((P_d(G_m),\mathcal{A}(V_m))_{m\in \mathbb{N}})_{\mathnormal{O_r}})
\rightarrow \lim\limits_{d\rightarrow \infty}K_*(C_u^*\left((P_d(G_m),\mathcal{A}(V_m))_{m\in \mathbb{N}}\right)_{\mathnormal{O_r}}) $$
is an isomorphism, as desired. The proof is complete.
\end{proof}
\par
Now we are ready to prove the main result of this section:
\par
\begin{proof}[Proof of Theorem \ref{twisted iso}:]
For any $r>0$, we define $\mathnormal{O_{m,r}}\subseteq \mathbb{R}_+\times H_m$ by
$$\mathnormal{O_{m,r}}=\bigcup\limits_{x \in G_m/N_m}\mathrm{Ball}(f(x),r),$$
where $f_m:G_m/N_m \rightarrow H_m$ is the coarse embedding, and
$$\mathrm{Ball}(f(x),r):=\{(t,h)\in \mathbb{R}_+\times H_m:t^2+\|h-f(x)\|^2<r\}.$$
Denote $O_r=(O_{m,r})_{m\in \mathbb{N}}$. By definition, we have $$C_u^*((P_d(G_m),\mathcal{A}(V_m))_{m\in \mathbb{N}})=\lim\limits_{r\rightarrow \infty}C_u^*((P_d(G_m),\mathcal{A}(V_m))_{m\in \mathbb{N}})_{\mathnormal{O_{r}}},$$
$$C_{u,L}^*((P_d(G_m),\mathcal{A}(V_m))_{m\in \mathbb{N}})=\lim\limits_{r\rightarrow \infty}C_{u,L}^*((P_d(G_m),\mathcal{A}(V_m))_{m\in \mathbb{N}})_{\mathnormal{O_{r}}}.$$
It is straightforward to check that
$$\lim\limits_{d\rightarrow \infty}\lim\limits_{r\rightarrow \infty}K_*(C_u^*((P_d(G_m),\mathcal{A}(V_m))_{m\in \mathbb{N}})_{\mathnormal{O_{r}}})
=\lim\limits_{r\rightarrow \infty}\lim\limits_{d\rightarrow \infty}K_*(C_u^*((P_d(G_m),\mathcal{A}(V_m))_{m\in \mathbb{N}})_{\mathnormal{O_{r}}}),$$
$$\lim\limits_{d\rightarrow \infty}\lim\limits_{r\rightarrow \infty}K_*(C_{u,L}^*((P_d(G_m),\mathcal{A}(V_m))_{m\in \mathbb{N}})_{\mathnormal{O_{r}}})
=\lim\limits_{r\rightarrow \infty}\lim\limits_{d\rightarrow \infty}K_*(C_{u,L}^*((P_d(G_m),\mathcal{A}(V_m))_{m\in \mathbb{N}})_{\mathnormal{O_{r}}}).$$
So it suffices to show that, for any $r_0>0$,
$$e_*: \lim\limits_{d \rightarrow \infty}K_*\Big(\lim\limits_{r<r_0,r\to r_0} C_{u,L}^*((P_d(G_m),\mathcal{A}(V_m))_{m\in \mathbb{N}})_{\mathnormal{O_{r}}} \Big)
\rightarrow \lim\limits_{d \rightarrow \infty}K_*\Big( \lim\limits_{r<r_0,r\to r_0} C_{u}^*((P_d(G_m),\mathcal{A}(V_m))_{m\in \mathbb{N}})_{\mathnormal{O_{r}}}  \Big)$$
is an isomorphism.
\par
Since the sequence $(G_m/N_m)_{m\in \mathbb{N}}$ has uniform bounded geometry and the family of coarse embeddings $f_m:G_m/N_m \rightarrow H_m$ are controlled by two common non-decreasing functions,  for any $r_0>0$ and $m\in \mathbb{N}$, there exist finitely many, say $k_0$ independent of $m$, mutually disjoint subsets $\{X_{m,k}\}_{k=1}^{k_0}$ of $G_m/N_m$ such that $G_m/N_m=\mathop{\sqcup}\limits_{k=1}^{k_0}X_{m,k}$, and for each $1\leq k\leq k_0$,
we have $d\big(f(x),f(x')\big)>3r_0$ for any two distinct elements $x, x'\in X_{m,k}$.
\par
For any $0<r<r_0$ and each $1\leq k\leq k_0$, let
$$\mathnormal{O_{m,r,k}}=\mathop{\bigcup}\limits_{x\in X_{m,k}}\mathrm{Ball}(f(x),r).$$
Then $\mathnormal{O_{m,r}}=\mathop{\cup}\limits_{k=1}^{k_0}\mathnormal{O_{m,r,k}}$. Moreover, every $\mathnormal{O_{m,r,k}}$ and
$$\mathnormal{O_{m,r,k}} \cap \big( \cup_{j=1}^{k-1}  \mathnormal{O_{m,r,j}}\big) $$
are $(G_m/N_m, r)$-separate for any $1< k\leq k_0$. By Proposition \ref{local iso} and a Mayer-Vietoris sequence argument, we have
$$e_*:\lim\limits_{d\rightarrow \infty}K_*(C_{u,L}^*((P_d(G_m),\mathcal{A}(V_m))_{m\in \mathbb{N}})) \rightarrow \lim\limits_{d\rightarrow \infty}K_*(C_u^*((P_d(G_m),\mathcal{A}(V_m))_{m\in \mathbb{N}}))$$
is an isomorphism. This completes the proof.
\end{proof}
	
\section{The Bott maps and proof of the main theorem}
In this section, we shall prove the main theorem of this paper. We shall construct an asymptotic morphism $\beta$ from
$\mathcal{S}\hat{\otimes}C_u^*((P_d(G_m))_{m\in \mathbb{N}})$ to $C_u^*((P_d(G_m),\mathcal{A}(V_m))_{m\in \mathbb{N}})$, which plays the role of the Bott map in the geometric analogue, as proved by G. Yu \cite{Yu00}, of infinite dimensional Bott periodicity introduced by Higson, Kasparov and Trout in \cite{HKT98}. This geometric analogue of
the infinite dimensional Bott periodicity is then used to reduce the coarse Novikov conjecture for the sequence of extensions of coarsely embeddable groups $(N_m)_{m\in \mathbb{N}}$ and $(G_m/N_m)_{m\in \mathbb{N}}$ to the twisted coarse Baum-Connes conjecture for the sequence $(G_m)_{m\in \mathbb{N}}$ with coefficients coming from a family of coarse embeddings of quotient groups $(G_m/N_m)_{m\in \mathbb{N}}$ to Hilbert spaces. We thus complete the proof of Theorem \ref{main result} by recalling Proposition \ref{twisted iso} of the previous section.
\par
Let $(1\to N_m\to G_m\to G_m/N_m\to 1)_{m\in \mathbb{N}}$ be a sequence of extensions of countable discrete groups such that $(N_m)_{m\in \mathbb{N}}$ and $(G_m/N_m)_{m\in \mathbb{N}}$
are coarsely embeddable into Hilbert space. Choose a family of proper length function $(l_{G_m})_{m\in \mathbb{N}}$ on $(G_m)_{m\in \mathbb{N}}$ to endow three  sequences of groups left  invariant metrics $d_{N_m}$, $d_{G_m}$ and $d_{G_m/N_m}$
 as before.
Let $\pi: G_m\rightarrow G_m/N_m$ be the quotient map, and let $f_m: G_m/N_m \rightarrow H_m$ be a coarse embedding of $G_m/N_m$ into a real Hilbert space $H_m$ for each $m\in \mathbb{N}$.
Recall that for each $n\in \mathbb{N}$ and $x\in G_m/N_m$, we define
$$W^m_n(x):=\mbox{span} \{f(y): y\in G_m/N_m, d_{G_m/N_m}(x,y)\leq n^2\}, $$
$$V_m:=W_m(x):=\cup_{n\in \mathbb{N}} W^m_n(x)$$
and, without loss of generality, assume that $V_m$ is dense in $H_m$. For each $d >0$, the map $f: G_m/N_m\to H_m$ is extended to the Rips complex
$f: P_d({G_m/N_m})\to H_m$. Choose a countable dense subset $Z^m_d$ of $P_d(G_m)$ such that $\pi (Z^m_d)$ is dense in $P_d(G_m/N_m)$, and $Z^m_{d} \subseteq Z^m_{d'}$ if
$d \leq d'$. For each $z\in Z^m_d$ with $\pi(z)=\sum_{x\in G_m/N_m} c_x x$, we define
$W^m_N(\pi(z))$ to be the Euclidean subspace of $V_m$ spanned by $W^m_n(x)$ for those $x$ such that $c_x\neq 0$.
\par
For each $m\in \mathbb{N}$ and $x\in Z^m_d$, and any finite-dimensional affine subspace $W_m$ of $V_m$ containing $f_m(\pi(x))$, recall that in
Definition \ref{HKT Bott map} we define a $*$-homomorphism
$$\beta_{W_m}(x):\mathcal{S}:=C_0(\mathbb{R}) \rightarrow \mathcal{A}(W_m)$$
by the formula
$$\beta_{W_m}(x)(g)=g\Big( X\hat{\otimes}1+1\hat{\otimes}C_{W_m,f_m(\pi(x))} \Big)$$
for all $g\in \mathcal{S}$, where $X$ is the degree one and unbounded multiplier of $\mathcal{S}$ defined by
$(X(g))(x)=xg(x)$ for each $g\in \mathcal{S}$ with compact support and all $x\in \mathbb{R}$,
and $C_{W_m,f_m(\pi(x))}$ is the Clifford algebra-valued function on $W_m$ defined by
$$C_{W_m,f_m(\pi(x))}(v)=v-f(\pi(x))\in W_m^{0} \subset \text{Cliff}(W_m^0)$$
for all $v\in W_m$. Taking inductive limit as in Definition \ref{HKT Bott map}, these maps $\beta_{W_m}(x)$ for all $W_m$ as above
induce a $*$-homomorphism
$$\beta_m(x):\mathcal{S} \rightarrow \mathcal{A}(V_m).$$
\par
\begin{definition}\label{Bott asymptotic morphisms}
Let $d\geq 0$. For each $t\in [1, \infty)$, we define a map
$$\beta_t:\mathcal{S}\hat{\otimes}\mathbb{C}_u[(P_d(G_m))_{m\in \mathbb{N}}] \rightarrow C_u^*((P_d(G_m),\mathcal{A}(V_m))_{m\in \mathbb{N}})$$
by the formula
$$\big( \beta_t(g\hat{\otimes}T) \big)^{(m)}(x,y)=\beta_m(x)(g_t)\hat{\otimes}T^{(m)}(x,y)$$
for all $g\in \mathcal{S}, T=(T^{(m)})_{m\in \mathbb{N}} \in \mathbb{C}_u[(P_d(G_m))_{m\in \mathbb{N}}]$ and $x, y\in Z^m_d$, where $g_t(\tau)=g(t^{-1}\tau)$ for all $\tau\in \mathbb{R}$,
and $\beta_m(x): \mathcal{S} \rightarrow \mathcal{A}(V_m)$ is the $*$-homomorphism associated to the inclusion of the zero-dimensional affine
space $f_m(\pi(x))$ into $V_m$  defined above.
\end{definition}
\par
Similarly, we have the following definition associated to the uniform localization algebras.	
\par
\begin{definition}
Let $d\geq 0$. For each $t\in [1, \infty)$, we define a map
$$(\beta_{u,L})_t:\mathcal{S}\hat{\otimes}\mathbb{C}_{u,L}[(P_d(G_m))_{m\in \mathbb{N}}] \rightarrow C^*_{u,L}((P_d(G_m),\mathcal{A}(V_m))_{m\in \mathbb{N}})$$
by the formula
$$\Big((\beta_{u,L})_t(g\hat{\otimes}h)\Big)(\tau)=\beta_t\Big( g\hat{\otimes}h(\tau) \Big)$$
for all $g\in \mathcal{S}$, $h\in \mathbb{C}_{u,L}[(P_d(G_m))_{m\in \mathbb{N}}]$ and $\tau\in [0, \infty)$, where
$\beta_t$ is defined in Definition \ref{Bott asymptotic morphisms} above.
\end{definition}
\par
We have the following result analogous to Lemma 7.6 in \cite{Yu00}.
\begin{lemma} \label{Bott}
For each $d\geq 0$, the maps $(\beta_t)_{t\geq 1}$ and $\big( (\beta_{u,L})_t \big)_{t\geq 1}$ extend to asymptotic morphisms
$$\beta:\mathcal{S}\hat{\otimes}C_u^*((P_d(G_m))_{m\in \mathbb{N}})\rightsquigarrow C_u^*((P_d(G_m),\mathcal{A}(V_m))_{m\in \mathbb{N}}),$$
$$\beta_{u,L}:\mathcal{S}\hat{\otimes}C_{L}^*((P_d(G_m))_{m\in \mathbb{N}})\rightsquigarrow C_{u,L}^*((P_d(G_m),\mathcal{A}(V_m))_{m\in \mathbb{N}}).$$
\end{lemma}
\par
\begin{proof}
To show that the maps $(\beta)_{t\geq 1}$ give rise to an asymptotic morphism, we first claim that
$$\|\beta(x)(g_t)-\beta(y)(g_t)\|$$
tends to $0$ uniformly on the band
$\{(x,y)\in Z^m_d\times Z^m_d : d_{P_d(G_m)}(x,y)\leq r\}_{m\in \mathbb{N}}$ for each fixed $g\in C_0(\mathbb{R})$ and $r> 0$.
\par		
Indeed, for each $m\in \mathbb{N}, r>0$ and $x,y\in Z^m_d$ with $d_{P_d(G_m)}(x,y)\leq r$, take $n\geq 1$ such that the finite-dimensional Euclidean subspace $W^m_n(\pi(x))$ of $V$
contains $0$, $f_m(\pi(x))$ and $f_m(\pi(y))$. By Definition \ref{HKT Bott map}, we have
$$\beta_m (x)(g_t)=\Big( \beta_{V_m,W^m_n(\pi(x))}\circ\beta_{W^m_n(\pi(x))}(x) \Big) (g_t),$$
$$\beta_m (y)(g_t)=\Big( \beta_{V_m,W^m_n(\pi(x))}\circ\beta_{W^m_n(\pi(x))}(y) \Big) (g_t)$$
where
$$\beta_{V_m,W^m_n(\pi(x))}:\mathcal{A}(W^m_n(\pi(x)))\rightarrow \mathcal{A}(V_m)$$
is the $*$-homomorphism associated to the inclusion $W^m_n(\pi(x)) \subset V_m$.
Thus it suffices to deal with
$$\Big( \beta_{W^m_n(\pi(x))}(x) \Big)(g_t)- \Big(\beta_{W^m_n(\pi(x))}(y)\Big)(g_t).$$
\par		
Denote $G_x=X\hat{\otimes}1+1\hat{\otimes}C_{W^m_n(\pi(x)),\pi(x)}$ and $G_y=X\hat{\otimes}1+1\hat{\otimes}C_{W^m_n(\pi(x)),\pi(y)}$. Note that $$C_{W^m_n(\pi(x)),\pi(x)}(v)-C_{W^m_n(\pi(x)),\pi(y)}(v)=(v-f_m(\pi(x)))-(v-f_m(\pi(y)))=f_m(\pi(y))-f_m(\pi(x))$$ for all $v\in W^m_n(x)$.
For the generators $g(s)=(s\pm i)^{-1}$ of $C_0(\mathbb{R})$, and for $t\in [1, \infty)$, we have that
\begin{equation*}
	\begin{aligned}
		\left\| \Big( \beta_{W^m_n(\pi(x))}(x) \Big) (g_t)- \Big( \beta_{W^m_n(\pi(x))}(y) \Big) (g_t) \right \|
		& = \left\|\Big(t^{-1}G_x \pm i\Big)^{-1}- \Big( t^{-1}G_y\pm i \Big)^{-1} \right\| \\
		&\leq t \Big\| (G_x\pm ti)^{-1} \Big\| \|G_x-G_y\| \Big\|(G_y\pm ti)^{-1} \Big\| \\
		&\leq t^{-1}\|f_m(\pi(y))-f_m(\pi(x))\|.
	\end{aligned}
\end{equation*}
\par
Since the family of  coarse embeddings $f_m: P_d({G_m/N_m})\to H_m$ are controlled two common non-decreasing functions and the quotient map $\pi: P_d(G_m)\to P_d(G_m/N_m)$ is contractive, we know that
$$\sup\limits_{m\in \mathbb{N}}\Big\{\|f_m(\pi(x))-f_m(\pi(y))\|\, x,y \in Z_d^m,  d_{P_d(G_m)}(x,y)\leq r\Big\}<\infty.$$
Hence,
$$\left\| \Big( \beta_{W^m_n(\pi(x))}(x) \Big) (g_t)- \Big( \beta_{W^m_n(\pi(x))}(y) \Big) (g_t) \right \|$$
tends to $0$ uniformly on $\{(x,y)\in Z^m_d\times Z^m_d : d_{P_d(G_m)}(x,y)\leq r\}_{m\in \mathbb{N}}$
for $g(s)=(s\pm i)^{-1}$ as $t\rightarrow \infty $.
However, by the Stone-Weierstrass theorem, the algebra generated by these two functions is dense in
$\mathcal{S}:=C_0(\mathbb{R})$. So we complete the
proof of the claim.
\par
Let $E^m$ be the Hilbert $C^*$-module over $\mathcal{A}(V_m)\hat{\otimes}K$ as defined in Definition \ref{twisted Roe algebra}.
For every $g\in \mathcal{S}$ and $T=(T^{(m)}))_{m\in  \mathbb{N}}\in \mathbb{C}_u[(P_d(G_m))_{m\in  \mathbb{N}}]$, and for $t\in [1, \infty)$, note that we have the following formula:
$$\beta_t\big( g\hat{\otimes} T \big) = N_{g_t}\cdotp \big( 1\hat{\otimes}T \big), $$
where $N_{g_t}:\bigoplus_{m}E^m\rightarrow \bigoplus_mE^m $ is defined to be the direct sum $N_{g_t}=\bigoplus_{m} N_{g_t}^m$, in which the bounded module homomorphisms  $N_{g_t}^m$ are defined by
$$N_{g_t}^m\left( \sum_{x\in Z^m_d}a_x[x] \right)=\sum_{x\in Z^m_d}\Big( \beta(x)(g_t)\hat{\otimes}1 \Big) a_x[x],$$
and $1\hat{\otimes} T: \bigoplus_m E^m\rightarrow \bigoplus_{m} E^m$ is the bounded module homomorphism defined by
$$(1\hat{\otimes} T)\bigoplus_m \left(\sum_{x\in Z^m_d}a_x[x]\right)=\bigoplus_m\left(\sum_{y\in Z^m_d}\Big(\sum_{x\in Z^m_d}\Big(1\hat{\otimes}T(y,x)\Big)a_x[x]\Big)[y]\right)$$
for all  $\sum_{x\in Z^m_d}a_x[x]\in E^m$. Since
$$\|N_{g_t}\|= \sup\limits_{x\in Z^m_d}\Big\| \beta(x)(g_t) \Big\|\leq \|g\|,$$
we have that
$$\Big\| \beta_t(g\hat{\otimes}T) \Big\|\leq \|g\|\|T\|.$$
It follows that $\beta_t$ extends to a bounded linear map from the algebraic tensor product $\mathcal{S}\hat{\odot} C_u^*((P_d(G))_{m\in \mathbb{N}})$ to $C_u^*((P_d(G_m),\mathcal{A}(V_m))_{m\in \mathbb{N}})$. Consequently, the family of maps $(\beta_t)_{t\geq 1}$ extends to an asymptotic morphism from the
maximal tensor product $\mathcal{S}\hat{\otimes}_{max}C_u^*((P_d(G))_{m\in \mathbb{N}})$ to $C_u^*((P_d(G_m),\mathcal{A}(V_m))_{m\in \mathbb{N}})$. Since $\mathcal{S}$ is nuclear, we conclude that $(\beta_t)_{t\geq 1}$ extends to an asymptotic morphism from  $\mathcal{S}\hat{\otimes}C_u^*((P_d(G))_{m\in \mathbb{N}})$ to $C_u^*((P_d(G_m),\mathcal{A}(V_m))_{m\in \mathbb{N}})$.	
The case for the localization algebras is similar. The proof is complete.
\end{proof}
\par
Note that the asymptotic morphisms
$$\beta:\mathcal{S}\hat{\otimes}C_u^*((P_d(G))_{m\in \mathbb{N}})\rightsquigarrow C_u^*((P_d(G_m),\mathcal{A}(V_m))_{m\in \mathbb{N}}),$$
$$\beta_{u,L}:\mathcal{S}\hat{\otimes}C_{L}^*((P_d(G))_{m\in \mathbb{N}})\rightsquigarrow C_{u,L}^*((P_d(G_m),\mathcal{A}(V_m))_{m\in \mathbb{N}})$$
induce homomorphisms on $K$-theory:
$$\beta_*:K_*(\mathcal{S}\hat{\otimes}C_u^*((P_d(G))_{m\in \mathbb{N}}))\rightarrow K_*(C_u^*((P_d(G_m),\mathcal{A}(V_m))_{m\in \mathbb{N}}))$$
$$(\beta_{u,L})_*:K_*(\mathcal{S}\hat{\otimes}C_{u,L}^*((P_d(G))_{m\in \mathbb{N}}))\rightarrow K_*(C_{u,L}^*((P_d(G_m),\mathcal{A}(V_m))_{m\in \mathbb{N}})).$$
\par

Finally we are ready to complete the proof of the main result of this paper.
\par
\begin{proof}[Proof of Theorem \ref{main result}:]	
Consider the following commutative diagram:
\begin{equation*}
	\xymatrix{
			\lim\limits_{d\rightarrow\infty}K_*(\mathcal{S}\hat{\otimes}C_{u,L}^*((P_d(G))_{m\in \mathbb{N}}))\ar[r]^{e_*}\ar[d]_{(\beta_{u,L})_*}^{\cong} &
                    \lim\limits_{d\rightarrow\infty}K_*(\mathcal{S}\hat{\otimes}C_u^*((P_d(G))_{m\in \mathbb{N}}))  \ar[d]_{\beta_*} \\
			\lim\limits_{d\rightarrow\infty}K_*(C_{u,L}^*((P_d(G_m),\mathcal{A}(V_m))_{m\in \mathbb{N}}))\ar[r]^{e_*}_{\cong} &
                    \lim\limits_{d\rightarrow\infty}K_*(C_u^*((P_d(G_m),\mathcal{A}(V_m))_{m\in \mathbb{N}})).
	}
\end{equation*}
We first show that $(\beta_{u,L})_*$ is an isomorphism for any $d\geq 0$. Note that the sequence of groups $(G_m)_{m\in \mathbb{N}}$ have uniform bounded geometry and the $K$-theory of all the above localization algebras is invariant under
the strong Lipschitz homotopy equivalence. Using  Mayer-Vietoris sequence argument for uniform localization algebras and uniform twisted localization algebras and induction on the dimension of skeletons of $P_d(G_m)$ for all $m$, the general case can be reduced to the $0$-dimensional case, namely, if $D_m\subset P_d(G_m)$ for each $m\in \mathbb{N}$ is a
$\delta$-separated subset (meaning that $d_{P_d(G_m)}(x,y)>\delta$ if $x\neq y\in D_m$) for some $\delta>0$, then
$$(\beta_{u,L})_*: K_*\Big(\mathcal{S}\hat{\otimes}C_{u,L}^*((D_m)_{m\in \mathbb{N}}) \Big)\to K_*\Big(C_{u,L}^*((D_m, \mathcal{A}(V_m))_{m\in \mathbb{N}} )\Big)$$
is an isomorphism. It is clear that this can be reformulated as
$$(\beta_{u,L})_*:\prod^u\limits_{\substack{x\in D_m \\ m\in \mathbb{N}}} K_*\Big(\mathcal{S}\hat{\otimes} C_{u,L}^*\big(\{x\}\big)\Big)\rightarrow
\prod^u\limits_{\substack{x\in D_m \\ m\in \mathbb{N}}} K_*\Big(C_{u,L}^*\big(\{x\},\mathcal{A}(V_m)\big)\Big),$$
which is an isomorphism by the celebrate infinite dimensional Bott periodicity theorem of Higson, Kasparov and Trout \cite{HKT98}.
It follows that the Bott map $(\beta_{u,L})_*$ in the commutative diagram above is an isomorphism even before the inductive limit in 
$d\to \infty$ is taken.
Since both $C_u^*((P_d(G))_{m\in \mathbb{N}})$ and $C_{u,L}^*((P_d(G))_{m\in \mathbb{N}})$ are trivially graded, by taking inductive limit and recalling Theorem \ref{twisted iso},
we conclude that the evaluation map on $K$-theory
$$e_*: \lim\limits_{d\rightarrow\infty}K_*\Big(C_{u,L}^*((P_d(G))_{m\in \mathbb{N}})\Big) \rightarrow \lim\limits_{d\rightarrow\infty}K_*\Big(C_u^*((P_d(G))_{m\in \mathbb{N}})\Big)$$
is injective. Now Theorem \ref{main result} follows from Proposition \ref{evaluation}. This completes the whole proof.
\end{proof}
	
\section*{Acknowledgements}
We would like to thank Jintao Deng and Guoliang Yu for valuable suggestions and helpful discussions. We also would like to thank Alexander Engel for helpful comments.

\vskip 1cm

\begin{itemize}

\item[] Qin Wang \\
Research Center for Operator Algebras,  and Shanghai Key Laboratory of Pure Mathematics and Mathematical Practice, School of Mathematical Sciences, East China Normal University, Shanghai, 200241, P. R. China. \quad
E-mail: qwang@math.ecnu.edu.cn

\item[] Yazhou Zhang \\
Research Center for Operator Algebras, School of Mathematical Sciences, East China Normal University, Shanghai, 200241, P. R. China.\quad
E-mail: 52185500010@stu.ecnu.edu.cn

\end{itemize}

\end{document}